\documentclass[11pt]{article}
\usepackage{bbm}
 \usepackage{amssymb}
\usepackage{amssymb, amsthm, amsmath, amscd}
\usepackage[usenames,dvipsnames]{color}

\newtheorem{thm}{Theorem}[section]
\newtheorem{lem}[thm]{Lemma}
\newtheorem{prop}[thm]{Proposition}
\newtheorem{cor}[thm]{Corollary}
\theoremstyle{definition}

\theoremstyle{definition}
\newtheorem{df}[thm]{Definition}
\theoremstyle{definition}
\newtheorem{rem}[thm]{Remark}
\newtheorem{nota}[thm]{Notation}
\theoremstyle{definition}
\newtheorem{exm}[thm]{Example}

\renewcommand{\phi}{\varphi}

\definecolor{purple}{RGB}{150,10,200} 

\newcommand{\CAs}{$C^*$-algebras}

\newcommand{\N}{\mathbb{N}}

\newcommand{\C}{\mathbb{C}}

\numberwithin{equation}{section}

\newcommand{\Aff}{\operatorname{Aff}}

\newcommand{\Her}{\mathrm{Her}}
\newcommand{\LAff}{\operatorname{LAff}}

\newcommand{\hm}{homomorphism}
\newcommand{\dt}{\delta}
\newcommand{\ep}{\varepsilon}

\newcommand{\td}{\tilde}

\newcommand{\la}{\langle}
\newcommand{\ra}{\rangle}
\newcommand{\andeqn}{\,\,\,{\rm and}\,\,\,}
\newcommand{\rforal}{\,\,\,{\rm for\,\,\,all}\,\,\,}
\newcommand{\CA}{$C^*$-algebra}
\newcommand{\SCA}{$C^*$-subalgebra}

\newcommand{\af}{{\alpha}}
\newcommand{\bt}{{\beta}}

\newcommand{\wtd}{\widetilde}

\newcommand{\diag}{{\rm diag}}

\newcommand{\wilog}{without loss of generality}
\newcommand{\Wlog}{Without loss of generality}

\newcommand{\beq}{\begin{eqnarray}}
\newcommand{\eneq}{\end{eqnarray}}
\newcommand{\tforal}{\,\,\,\text{for\,\,\,all}\,\,\,}
\newcommand{\tand}{\,\,\,\text{and}\,\,\,}

\usepackage{amsfonts}
\usepackage{mathrsfs}
\usepackage{textcomp}
\usepackage[all]{xy}

\title{Tracial approximate divisibility  and stable rank one}
\author{Xuanlong Fu, Kang Li, and Huaxin Lin}
\date{September 2021} 

\begin{document}

\maketitle

\begin{abstract}
{{In this paper, we show that every separable simple tracially approximately divisible 
\CA\ has strict comparison, and, it is either purely infinite or has stable rank one.
As a consequence, we show that every (non-unital) finite  {{simple}}
${\cal Z}$-stable \CA\, has stable rank one. }}
\end{abstract}

\section{Introduction}

Approximate divisibility for \CA s was introduced in \cite{BKR}  in the study of noncommutative tori 
{{following the earlier work of M.~R\o rdam (see \cite{Rordam-1991-UHF}  and
\cite{Rordam-1992-UHF2}).}}
It is shown in {{Theorem 1.4}} of  \cite{BKR} that a {{unital separable simple}}  \CA\ $A$ which is approximately divisible
{{has}} strict comparison, and is either purely infinite or has stable rank one.  

Tracial approximation was introduced in the Elliott program of classification for simple \CA s at the end 
of last century  {{(see, for example, \cite{LinTAF}, \cite{LinTRK} and \cite{LinDuke}).}}
The  term of tracially  
approximate divisibility appeared at  the same time as the study of simple \CA s of tracial rank one (see {{Definition}} 5.3 
and the proof of Theorem 5.4 of \cite{Lintrk1}). 
It was shown, for example, that every unital non-elementary simple \CA\, with tracial rank at most one is tracially approximately 
divisible.
A more  general version of tracially approximate divisibility  was  given in  
Definition 5.2 of \cite{FLII}.
Similar variations 
of tracially approximate divisibility
 also 
occurred (see  Definition \ref{TAD} below, 
and also in Definition 10.1 of \cite{EGLN}).
 A concept  with the  same nature was also given in \cite{HO}
which was called  tracially ${\cal Z}$-absorbing
(see also \cite{CLS} and \cite{AGJP17}).
{{As}} a continuation of \cite{FLII} (and also of \cite{FL}),
we {{first}} show that these notions of tracially 
approximate divisibility are {{all}} equivalent for (not necessarily unital) 
non-elementary
separable simple 
\CA s (see Theorem \ref{TTAD==}).


With the same spirit  of \cite{BKR}, we {{also}} show that a separable simple \CA\, which is tracially 
approximately divisible  has strict comparison, and is either purely infinite or has stable rank one (see Theorem \ref{TAD-SC} and Corollary \ref{Ctadsr1}). Moreover, we  
show that if $A$ is a  {{non-elementary}} separable simple 
\CA\ 
  {{which is 
 tracially approximately divisible,  then}} its Cuntz semigroup can
be written
as ${\rm Cu}(A)=(V(A)\setminus \{0\})\sqcup \LAff_+({\wtd{QT}}(A))$
(see Theorem \ref{TLAFF}  and Remark \ref{RCuntz}).

We would like to point out that the Jiang-Su algebra ${\cal Z}$ is not approximately divisible 
since it has no non-zero projection other than the unit
(\cite{JS1999}).
However, by Theorem 5.9 of \cite{FLII}, every simple \CA\ which can be 
{{essentially}}
tracially approximated 
by separable
${\cal Z}$-stable \CA s (see Definition 3.1 of \cite{FLII}  \and Definition \ref{DTrapp}  below) is tracially approximately divisible. 
In particular, simple ${\cal Z}$-stable 
\CA s are tracially approximately divisible. 
In {{Example}} \ref{REx1}, we observe  that there are a whole set of non-nuclear
separable simple \CA s which are tracially approximately divisible.   Since every unital
simple \CA\, which has tracial {{rank}} 
zero is tracially approximately divisible, 
there exist tracially approximately divisible \CA s which are not ${\cal Z}$-stable by \cite{NW}.

Gong, Jiang and Su showed in \cite{GJS00} that a unital simple ${\cal Z}$-stable \CA\ $A,$ i.e., $A\cong A\otimes {\cal Z},$ 
is either purely infinite, or is stably finite, and has weakly unperforated $K_0(A).$
In \cite{Rordam-2004-Jiang-Su-stable}, R\o rdam showed that 
a unital simple ${\cal Z}$-stable \CA\ $A$ 
is either purely infinite, or has 
stable rank one, and  has almost unperforated Cuntz-semigroup. If $A$ is {{a  
separable
simple
 ${\cal Z}$-stable \CA\,}} and contains a non-zero projection $p$, then $pAp$ is also ${\cal Z}$-stable
(\cite[Corollary 3.1]{TomsW07}). 
One then quickly concludes that $A$ has stable rank one if it is finite.
In \cite{Rlz}, L. Robert showed  that every stably 
projectionless  simple \CA\ $A$ which is ${\cal Z}$-stable 
has almost 
stable
rank one.  It left open whether  a 
stably projectionless simple  
${\cal Z}$-stable \CA\ has  
stable rank one (see Question 3.5 of \cite{Rlz}).  As a by-product,  we show that
every finite 
simple ${\cal Z}$-stable \CA\ always has stable rank one (Corollary \ref{Cor-Z-stable-F817}). In particular, we answer  Robert's question affirmatively. 
Some {{applications and examples}} to dynamical  systems {{can be found}} at the end of this paper.
{{We also refer the reader to the recent papers  \cite{T20} and \cite{APRT} for 
further  related results about \CA s of stable rank one.}} 


The paper is organized as follows.
Section  2 is a preliminary.  Section 3 discusses  the 
so-called Cuntz-null sequences.
In Section 4,   we discuss 
several variations of tracial approximate divisibility
and 
we show in Theorem \ref{TTAD==} that they are actually all equivalent. 
In Section 5, we show that a separable simple tracially approximately
divisible \CA\ has strict comparison {{(see Theorem \ref{TAD-SC})}}. Moreover, we also show that the canonical map from {{the purely non-compact elements in}} Cuntz semigroup ${\rm Cu}(A)$ to the set of strictly positive 
lower semi-continuous affine functions  in 
$\LAff_+({\wtd{QT}}(A))$ is an order-isomorphism (see Theorem \ref{TLAFF}). In Section 6,  
we show that a separable simple 
tracially approximately divisible \CA\ 
is either purely infinite, 
or has stable rank one (see Corollary \ref{Ctadsr1}). We end Section 6 by showing that every (non-unital) simple ${\cal Z}$-stable \CA\, is either purely infinite or has stable rank one (see Corollary \ref{Cor-Z-stable-F817}). Finally, we include some examples in Section 7.

\hspace{0.2in}

{\bf Acknowledgements}\hspace{0.1in}
This research
began when 
the first and the third named authors stayed
in the Research Center for Operator Algebras in East China Normal University
in the summer of 2019.
    The first and third named  authors acknowledge the support by the Center
 which is in part supported  by NNSF of China (11531003)  and Shanghai Science and Technology
 Commission (13dz2260400),
 and  Shanghai Key Laboratory of PMMP.
 The second named author was supported by the Internal KU Leuven BOF project C14/19/088.
The third named author was also supported by NSF grants (DMS 1665183 and DMS 1954600).
The first and the third named authors  would like to thank George Elliott for many 
helpful conversations and 
comments.

\section{Preliminary}
In this paper,
the set of all positive integers is denoted by $\N.$ 
{{The set of all compact operators on a separable 
infinite-dimensional Hilbert 
space is denoted by ${\cal K}.$}}

\begin{nota}
Let  $A$ 
be a normed space and ${\cal F}\subset A$ be a subset. Let  $\epsilon>0$.
Let $a,b\in A,$
we  write $a\approx_{\epsilon}b$ if
$\|a-b\|< \epsilon$.
We write $a\in_\ep{\cal F},$ if there is $x\in{\cal F}$ such that
$a\approx_\ep x.$
\end{nota}

\begin{nota}
Let $A$ be a $C^*$-algebra and
let $S\subset A$ be a subset of $A.$
Denote by
${\rm Her}_A(S)$ (or just $\Her(S),$ when $A$ is clear)
the hereditary $C^*$-subalgebra of $A$ generated by $S.$
Denote by $A^{\bf 1}$ the closed unit ball of $A,$
by $A_+$ the set of all positive elements in $A.$
Put $A_+^{\bf 1}:=A_+\cap A^{\bf 1}.$
Denote by $\wtd A$ the minimal unitization of $A.$
When $A$ is unital, denote by $GL(A)$ the group of invertible elements of $A,$
 and denote by  $U(A)$ the unitary group of $A.$
 {{Let  ${\rm Ped}(A)$ denote
the Pedersen ideal of $A$ and ${\rm Ped}(A)_+:= {\rm Ped}(A)\cap A_+$.  Denote by $T(A)$ the tracial state space of $A.$}}
\end{nota}
\begin{df}
Let $A$ and $B$ be \CA s and 
$\phi: A\rightarrow B$ be a  linear map.
The map $\phi$ is positive, if $\phi(A_+)\subset B_+.$  
The map $\phi$ is completely positive contractive, abbreviated as c.p.c.,
if $\|\phi\|\leq 1$ and  
$\phi\otimes \mathrm{id}: A\otimes M_n\rightarrow B\otimes M_n$
are positive for all $n\in\mathbb{N}.$ 
A c.p.c.~map $\phi: A\to B$ is called order zero, if for any $x,y\in A_+,$
$xy=0$ implies $\phi(x)\phi(y)=0.$

In what follows, $\{e_{i,j}\}_{i,j=1}^n$ (or just $\{e_{i,j}\},$ if there is no confusion) is a system of matrix unit for $M_n,$ and, $\iota\in C_0((0,1])$ 
is the identity function on $(0,1],$  i.e., $\iota(t)=t$ for all $t\in (0,1].$
\end{df}
\begin{df}
A   \CA\ $A$ is said to have stable rank one
(\cite{Rie83}), 
if  $\widetilde A=\overline{GL(\widetilde A)},$
i.e.,  $GL(\widetilde A)$ is dense in $\widetilde A.$
A  \CA\ $A$ is said to have 
almost stable rank one (\cite{Rlz}), 
if for any hereditary $C^*$-subalgebra $B\subset A,$
$B\subset  \overline{GL(\widetilde B)}.$
\end{df}

\begin{nota}
Let $\epsilon>0.$ Define a continuous function
$f_{\epsilon}: [0,+\infty)
\rightarrow [0,1]$ by
$$
f_{\epsilon}(t)=
\left\{\begin{array}{ll}
0  &t\in {{[0,\epsilon/2]}},\\
1  &t\in [\epsilon,\infty),\\
\mathrm{linear } &{t\in[\epsilon/2, \epsilon].}
\end{array}\right.
$$

\end{nota}

\begin{df}\label{Dcuntz}
Let $A$ be a \CA\
and  let $M_{\infty}(A)_+:=\bigcup_{n\in\mathbb{N}}M_n(A)_+$.
For $x\in M_n(A),$
we identify $x$ with ${\rm diag}(x,0)\in M_{n+m}(A)$
for all $m\in \N.$
Let $a\in M_n(A)_+$ and $b\in M_m(A)_+$.
We may write  $a\oplus b:=\mathrm{diag}(a,b)\in M_{n+m}(A)_+.$
If $a, b\in M_n(A),$
we write $a \lesssim b$ if there are
$x_i\in M_n(A)$
such that
$\lim_{i\rightarrow\infty}\|a-x_i^*bx_i\|=0$.
We write $a \sim b$ if $a \lesssim b$ and $b \lesssim a$ hold. 
The Cuntz relation $\sim$ is an equivalence relation.
Set $W(A):=M_{\infty}(A)_+/\sim$.
Let $\la a\ra$ denote the equivalence class of $a$. 
We write $\la a\ra\leq \la b\ra $ if $a \lesssim  b$.
$(W(A),\leq)$ is a partially ordered abelian semigroup.
Let ${\rm Cu}(A)=W(A\otimes {\cal K}).$ 
$W(A)$ (resp. ${\rm Cu}(A)$) is called almost unperforated,
if for any $\la a \ra, \la b\ra\in W(A)$ (resp. ${\rm Cu}(A)$),
and for any $k\in\N$,
if $(k+1)\la a\ra \leq k\la b\ra$,
then $\la a \ra \leq \la b\ra$
(see \cite{Rordam-1992-UHF2}).
Denote by  $V(A)$  the subset of those elements in $W(A)$ represented by projections.
\end{df}

\begin{rem}
It is known to experts that $W(A)$ is almost unperforated  is equivalent to say that ${\rm Cu}(A)$ is almost unperforated.
To see this briefly, let $a, b\in (A\otimes {\cal K})_+$ such that $(k+1)\la a\ra \le k\la b\ra.$ 
Let $\{e_{i,j}\}$ be the system of matrix units for ${\cal K}$ and $E_n=\sum_{i=1}^n1_{\td A}\otimes e_{i,i}$ and let 
$\ep>0.$   
Note that $E_naE_n\in M_n(A)_+$ for all $n\in \N.$
Moreover, $a\approx_{\ep/8} E_naE_n$ for some large $n\in \N.$ 
It follows  from Proposition 2.2 of \cite{Rordam-1992-UHF2} 
that $(a-\ep)_+\lesssim (E_naE_n-\ep/4)_+$
and $(E_naE_n-\ep/4)_+\lesssim (a-\ep/8)_+.$
By Proposition 2.4 of of \cite{Rordam-1992-UHF2}
, there exists $\dt>0$ such that 
$(k+1)\la (a-\ep/8)_+\ra \le k\la (b-\dt)_+\ra.$ Repeating R\o rdam's results 
(\cite{Rordam-1992-UHF2}),
one obtains 
that $\la (b-\dt)_+\ra \le  
{{\la E_mbE_m\ra}}$ for some even larger $m.$
Now one has $(k+1)\la (E_naE_n-\ep/4)_+\ra \le k\la E_mbE_m\ra.$ Since $W(A)$ is almost unperforated,
$(a-\ep)_+\lesssim (E_naE_n-\ep/4)_+\lesssim E_mbE_m.$ Then,
$(a-\ep)_+\lesssim E_mbE_m\lesssim b.$ It follows that $a\lesssim b.$
{{Therefore $W(A)$ is almost unperforated implies ${\rm Cu}(A)$
is almost unperforated.}} 

{{To see the converse, just notice that 
$A$ is a hereditary $C^*$-subalgebra of $A\otimes {\cal K},$
$\la a\ra\leq \la b\ra $ in ${\rm Cu}(A)=W(A\otimes {\cal K})$ 
implies  
$\la a\ra\leq \la b\ra $ in $W(A).$}}

\end{rem}

\begin{df}\label{Dqtr}
{{Let $A$ be a \CA. 
A densely  defined 2-quasi-trace  is a 2-quasi-trace defined on ${\rm Ped}(A)$ (see  Definition II.1.1 of \cite{BH}). 
Denote by ${\widetilde{QT}}(A)$ the set of densely defined 2-quasi-traces 
on 
$A\otimes {\cal K}.$  
 In what follows we will identify 
$A$ with $A\otimes e_{1,1},$ whenever it is convenient. 
Let $\tau\in {\widetilde{QT}}(A).$  Note $\tau(a)\not=\infty$ for any $a\in {\rm Ped}(A)_+\setminus \{0\}.$
}}

Note, for each $a\in ({{A}}
\otimes {\cal K})_+$ and $\ep>0,$ $f_\ep(a)\in {\rm Ped}(A\otimes {\cal K})_+.$ 
Define 
\beq
d_\tau(a)=\lim_{\ep\to 0}\tau(f_\ep(a))\rforal \tau\in {\widetilde{QT}}(A).
\eneq

A simple \CA\ 
$A$ 
is said to have (Blackadar's) strict comparison, if, for any $a, b\in (A\otimes {\cal K})_+,$ 
one has 
$a\lesssim b,$ if 
\beq
d_\tau(a)<d_\tau(b)\rforal \tau\in {{{\widetilde{QT}}(A)\setminus \{0\}.}}
\eneq

Let $A$ be a simple \CA. By Proposition 3.2 of \cite{Rordam-2004-Jiang-Su-stable}  (and Proposition 6.2 of \cite{ERS}), if ${\rm Cu}(A)$  is almost unperforated then $A$ has strict comparison (see also Proposition 4.2  of \cite{ERS}).
 
We endow ${\widetilde{QT}}(A)$ 
{{with}} the topology  in which a net 
${{\{}}\tau_i{{\}}}$ 
 converges to $\tau$ if 
${{\{}}\tau_i(a){{\}}}$ 
 converges to $\tau(a)$ for all $a\in 
 {\rm Ped}(A)$ 
 (see also (4.1) on page 985 of \cite{ERS}).
 
 Let $A$ be a simple \CA.   Note that, if $\tau$ is a lower semicontinuous quasitrace on $A\otimes {\cal K}$ defined in \cite{ERS},
and if $\tau(a)<\infty$ {{for}} some $a\in {\rm Ped}(A)_+\setminus \{0\},$ 
then $\tau(c)\in \C$ for all $c\in {\rm Ped}(A).$ In other words, 
$\tau\in {\widetilde{QT}}(A).$ If $\tau(a)=\infty$ 
for some  $a\in {\rm Ped}(A)_+\setminus \{0\},$ then, in this case, 
$\tau(c)=\infty$ for all $c\in {\rm Ped}(A)_+\setminus \{0\}.$

 Choose any {{$e\in {\rm Ped}(A)_+\setminus \{0\}.$}} Put $T_e=\{\tau\in {\wtd{QT}}(A): \tau(e)=1\}.$
 Then $T_e$ is compact (see Theorem 4.4 of \cite{ERS}, note that
$T_e=\{\tau\in QT_2(A): \tau(e)=1\}$  is a closed subset in the compact  space  $QT_2(A),$  which is used in \cite{ERS}).
 Suppose that ${\wtd{QT}}(A)\not=\{0\}.$ 
 Since $A$ is simple, if $\tau\in {\wtd{QT}}(A)\setminus \{0\},$ then $\tau(e)>0.$ 
 Note that, for any $\tau\in {\wtd{QT}}(A)\setminus \{0\},$ $\tau(\cdot)/\tau(e)\in T_e.$
 In other words, $\tau(a)<\tau(b)$ for all $\tau\in {\wtd{QT}}(A)\setminus \{0\}$ if and only if 
 $\tau(a)<\tau(b)$ for all $\tau\in T_e.$

\end{df}

\begin{df}\label{DLAFFQT}
Let $\Aff_+({\wtd{QT}}(A))$ be the set of all continuous affine functions $f$ on ${\wtd{QT}}(A)$
such that $f(\tau)>0$ for all $\tau\in {\wtd{QT}}(A)\setminus \{0\}$ and $f(0)=0,$
and, the zero function.

Let $\LAff_+({\wtd{QT}}(A))$ be the set of  those 
lower semi-continuous affine functions $f:{\wtd{QT}}(A)\to [0, \infty]$ such that 
there exists an increasing  sequence of functions $f_n\in \Aff_+({\wtd{QT}}(A))$
such that 
$f(\tau)=\lim_{n\to\infty}f_n(\tau)$ for all $\tau\in {\wtd{QT}}(A).$ 
%

The canonical map from ${\rm Cu}(A)$ to $\LAff_+({\wtd{QT}}(A))$ is defined as follows:
for each $a\in (A\otimes {\cal K})_+,$ the map $\la a\ra \to \widehat{\la a\ra}$ is defined 
by $\widehat{\la a\ra}(\tau)=d_\tau(a)$ for all $\tau\in {\wtd{QT}}(A).$
\end{df}

\section{Cuntz-null Sequences}

\begin{df}
Let $A$ be a \CA. 
A bounded sequence $\{a_n\}$ in $A$ is said to be {{Cuntz-}}null, if for any $a\in A_+\backslash\{0\}$ and any $\ep>0,$ 
there is $n_0\in\mathbb{N}$ such that $f_\ep(a_n^*a_n)\lesssim a$ for all $n\geq n_0.$ 

Let $l^\infty(A)$ be the \CA\ of bounded sequences of $A.$
{{Recall}} that 
$c_0(A):=\{\{a_n\}\in l^\infty(A): \lim_{n\to\infty}\|a_n\|=0\}$ 
is a (closed) two-sided ideal of $l^\infty(A).$ 
Let $A_\infty:=l^{\infty}(A)/c_0(A).$
Let $\pi_\infty: l^\infty(A)\to 
A_\infty$ be the quotient map. 
We view $A$ as a subalgebra of $l^\infty(A)$ 
via the canonical map $a\mapsto\{a,a,,...\}$ for all $a\in A.$
In what follows, we will identify $a$ with the constant sequence $\{a,a,...,\}$ in 
$l^\infty(A)$ without further warning.
Denote by     
$
N_{cu}(A)
$
(or just $N_{cu}$)
 the set of all Cuntz-null sequences  in $l^\infty(A).$
\end{df}

\begin{rem}
For a free (ultra)filter $\omega$ on $\mathbb{N},$
we may similarly define $\omega$-Cuntz-null sequences as follows: 
the set of those  $\{a_n\}\in l^\infty(A)$ such that, for any $a\in A_+\backslash\{0\}$
and any $\ep>0,$
there is $W\in \omega$ satisfying $f_{\ep}(a_n^*a_n)\lesssim a$ for all 
$n\in W.$
Similar results in this section also works for 
$\omega$-Cuntz-null sequences.
But we will not  explore this further in this paper. 
\end{rem}

 \begin{prop}
 \label{Infinitesimal-ideal}
  Let $A$ be a \CA\, {{and}} $B\subset A$ be a \SCA. 
 Let 
 $$
 I:=\{a\in A: a^*a\lesssim b
\rforal b\in B_+\backslash\{0\}\}.
$$ 
If $B$ has no one-dimensional hereditary \SCA s,
then $I$  is a  closed two-sided ideal of $A.$
 \end{prop}
 
\begin{proof}
First, let us show that  $I$ is a $*$-invariant linear space.
To see this,  let $a\in I.$ If $\lambda\in \C$ 
and $b\in B_+\setminus \{0\},$
then
$|\lambda|^2 a^*a\lesssim a^*a\lesssim b. $  Thus $\lambda a\in I.$ 
{{Since $a^*a\sim aa^*,$}} we also have $a^*\in I.$
Now
let $a_1,a_2\in I$ and  $b\in B_+\backslash\{0\}.$ {{By}} our assumption, 
$\Her_B(b)\ncong\C.$ Then there exist $b_1,b_2\in \Her_B(b)_+\backslash\{0\}$
such that $b_1b_2=0.$
By  {{the}} definition of $I,$
we  have $a_i^*a_i\lesssim b_i,$ ($i=1,2$).
Then $(a_1+a_2)^*(a_1+a_2)\le 2(a_1^*a_1+a_2^*a_2)\lesssim b_1+b_2\lesssim b.$
It follows  $a_1+a_2\in I.$   Consequently,
$I$ is a $*$-invariant linear space. 

{{Next, let us show}} that  {{$I$}} is a two-sided (algebraic) ideal.  Let $a\in I$ and  $x\in A.$ 
For any $b\in {{B_+}}\backslash\{0\},$ we have ${{(ax)^*ax}}=x^*a^*ax
 \lesssim a^*a\lesssim b.$ 
{{Similarly,}}
$(xa)^*xa=a^*x^*xa\lesssim a^*a\lesssim b.$  Thus $ax,xa\in I.$ 
This shows that  $I$ is a two-sided (algebraic) ideal of $A.$

{{It remains to show}} that $I$ is norm closed.  Let ${{\{a_n\}}}\subset I$ and $a\in A$ such that 
$\|a_n-a\|\to 0$ $(n\to\infty).$ 
Let $b\in {{B_+}}\backslash\{0\}.$ For any $\ep>0,$
there is $m\in\N$ such that  $a^*a \approx_{\ep/2}a_m^*a_m.$
Since $a_m\in I,$ we have ${{a_m^*}}a_m\lesssim b.$
{{Thus}} there exist
 $r\in A$ such that 
$a_m^*a_m\approx_{\ep/2}rbr^*.$ Then $a^*a\approx_{\ep}rbr^*.$
Since $\ep$ is arbitrary,
we have $a^*a\lesssim b.$ Since $b\in {{B_+}}\backslash\{0\}$ is arbitrary, 
we have $a\in I.$ Thus $I$ is norm closed as desired.
\end{proof}


\begin{prop}
\label{seq-cu-prpo1}
{{An element}}
$\{a_n\}$ in  $l^\infty(A)$ is a Cuntz-null sequence if and only if 
$\pi_\infty(\{{{a_n^*}}a_n\})\lesssim \pi_\infty(a)$
for all $a\in A_+\backslash\{0\}.$ 
\end{prop} 
\begin{proof}
{{To see the ``if" part, let us assume that}} 
$\{a_n\}\in l^\infty(A)
$ {{has the property that,}}
for any $a\in A_+\backslash\{0\},$ 
\beq
\pi_\infty(\{a_n^*a_n\})
=\pi_\infty(\{a_n\})^*\pi_\infty(\{a_n\})
\lesssim
\pi_\infty(a).
\eneq
For any $\ep>0,$ there is $r=\pi_\infty(\{r_n\})\in l^\infty(A)/c_0(A)$
such that $\|\pi_\infty(\{a_n^*a_n\})-r^*\pi_\infty(a)r\|<\ep/2.$
Then there is $n_0\in\mathbb{N}$ such that 
$\|a_n^*a_n-r_n^*ar_n\|<\ep/2$ for all $n\geq n_0.$ 
By \cite[Proposition 2.2]{Rordam-1992-UHF2}, 
we have $f_\ep(a_n^*a_n)\lesssim r_n^*ar_n\lesssim a$ for all $n\geq n_0.$
{{So $\{a_n\}\in N_{cu}(A).$}}

Conversely, we  assume that $\{a_n\}\in N_{cu}(A).$ Let $a\in A_+\backslash\{0\}$ and  $\ep>0.$ 

{{Choose}} $\dt>0$  such that $f_\dt(a)\neq 0.$
Since $\{a_n\}$ is a Cuntz-null sequence, there is $n_1\in\mathbb{N}$ satisfying 
$f_{\ep/2}(a_n^*a_n)\lesssim f_\dt(a)$ for all $n\geq n_1.$
By \cite[Proposition 2.4 (iv)]{Rordam-1992-UHF2}, 
for each $n\geq n_0,$ there is $r_n\in A$ such that $f_{\ep}(a_n^*a_n)=r_n^*f_\dt(a)r_n.$
Note that $\|f_\dt(a)^{1/2}r_n\|=\|r_n^*f_\dt(a)r_n\|^{1/2}=\|f_{\ep}(a_n^*a_n)\|$ 
is bounded for all $n\geq n_1$. 
For $n\in\N,$ let $s_n=f_\dt(a)^{1/2}r_n$ if $n\geq n_1,$ and let $s_n=0$ if $n< n_1.$
Then  $s=\{s_n\}\in l^\infty(A).$ 
{{Moreover,}}
$f_{\ep}(a_n^*a_n)=s_n^*f_{\dt/2}(a)s_n$ for all $n\geq n_1.$ Then
\beq
\|f_{\ep}(\pi_\infty(\{a_n^*a_n\}))-s^*f_{\dt/2}(\pi_\infty(a))s\|
&=&
\|\pi_\infty(f_{\ep}(\{a_n^*a_n\})-s^*f_{\dt/2}(\pi_\infty(a))s\|
\\
&\leq&
\sup_{n\geq n_1}
\|f_{\ep}(a_n^*a_n)-s_n^*f_{\dt/2}(a)s_n\|=0.
\eneq
Therefore 
$
f_{{\ep}}(\pi_\infty(\{a_n^*a_n\}))
\lesssim
f_{\dt/2}(\pi_\infty(a))
\lesssim
\pi_\infty(a).
$
It follows  
$\pi_\infty(\{a_n^*a_n\})\lesssim
\pi_\infty(a)$ as desired. 
\end{proof}
\begin{prop}\label{Ncuideal}
If $A$ is a \CA\ which
has no one-dimensional hereditary {{\SCA s,}} then 
$N_{cu}(A)$ is a closed two-sided ideal of $l^\infty(A)$
and $c_0(A)\subset N_{cu}(A).$
If $A$ is  a non-elementary separable simple \CA, then $c_0(A)\not=N_{cu}(A).$
\end{prop}
\begin{proof}
Let $J:=\{x\in l^\infty(A)/c_0(A): {{x^*x\lesssim  a\rforal a\in {{\pi_\infty(A)_+}}}}\backslash\{0\}\}.$ 
Since 
$A$ 
has no one-dimensional hereditary {{\SCA,}} by 
Proposition \ref{Infinitesimal-ideal}, 
$J$ is a norm closed two-sided ideal of 
$l^\infty(A)/c_0(A).$ 
Then, {{by Proposition \ref{seq-cu-prpo1},}} $N_{cu}(A)=\pi_\infty^{-1}(J)$ is a norm closed two-sided ideal of $l^\infty(A).$
Moreover, $c_0(A)=\pi^{-1}_\infty(0)\subset 
\pi_\infty^{-1}(J)=N_{cu}(A).$

Suppose now that $A$ is non-elementary, separable and simple, by \cite[Lemma 4.3]{FL}, 
there exists a sequence 
$\{s_n\}\subset A_+$ 
with $\|s_n\|=1$
such that, for any $a\in {{A}}_+\setminus \{0\},$ there exists $n_0\ge 1$ such that 
$s_n\lesssim a$ for all $n\geq n_0.$  In other words, $\{s_n\}\in N_{cu}(A).$ 
Note that $\{s_n\}\not\in c_0(A).$
\end{proof}

\begin{df}
\label{seq-def-F812-1}
Let $A$ be a \CA\ 
which has no one-dimensional hereditary \SCA s.
Let $A_{cu}:=l^\infty(A)/N_{cu}(A)$ and 
$\pi_{cu}: l^\infty(A)\to A_{cu}$ be the quotient map. 
Let 
$\pi_\infty(A)':=\{x\in A_\infty: xa=ax\mbox{ for all }a\in \pi_\infty(A)\}.$
Let 
$\pi_\infty(A)^\bot:=\{x\in A_\infty: xa=0=ax\mbox{ for all }a\in \pi_\infty(A)\}.$
Let 
$\pi_{cu}(A)':=\{x\in A_{cu}: xa=ax\mbox{ for all }a\in \pi_{cu}(A)\}.$
Let 
$\pi_{cu}(A)^{\bot}:=\{x\in A_{cu}: xa=0=ax\mbox{ for all }a\in \pi_{cu}(A)\}.$
Recall that $c_0(A)\subset N_{cu}(A).$  Denote by
$\pi: A_\infty\to A_{cu},$
$\pi_\infty(\{a_n\})\mapsto \pi_{cu}(\{a_n\})$ 
 the well-defined quotient map. 
Moreover, $\pi:A_\infty\to A_{cu}$ 
induces  canonical {{maps}} ${{\pi'}}: \pi_\infty(A)'\to \pi_{cu}(A)'$
and  $\pi^\bot: \pi_\infty(A)^\bot\to \pi_{cu}(A)^\bot.$
\end{df}

\begin{rem}
It is obvious that $\pi_\infty(A)^\bot$ is a closed two-sided ideal of $\pi_\infty(A)',$
and
 $\pi_{cu}(A)^\bot$ is a closed two-sided ideal of $\pi_{cu}(A)'.$
\end{rem}

\begin{prop}
\label{equiv-ideal-F812-1}
Let $A$ be {{a  non-elementary  separable  simple}} \CA\, with  
${\wtd{QT}}(A)\not=\{0\}.$ 
Let $e\in {\rm Ped}(A)_+\setminus \{0\}$ and 
let $T_e=\{\tau\in {\wtd{QT}}(A): \tau(e)=1\}.$ 
Define 
\beq
&&\hspace{-0.2in}I_{T,0}=\{\{x_n\}\in l^\infty(A): \lim_{n\to\infty}\sup\{\tau(x_n^*x_n):\tau\in T_e\}=0\}\tand\\
&&\hspace{-0.2in}I_{T}:=\{\{x_n\}\in l^\infty(A): \lim_{n\to\infty}\sup\{\tau((x_n^*x_n-\ep)_+):\tau\in T_e\}=0\rforal \ep>0\}.
\eneq
Suppose that $A$ has strict comparison. 
Then 
\beq
\label{inc-F811-1}
\overline{I_{T,0}}\subset I_T=
N_{cu}(A).
\eneq
Moreover, if $A={\rm Ped}(A),$ 
then 
$I_{T,0}=N_{cu}(A)=I_T.$
\end{prop}

\begin{proof}
To see
\eqref{inc-F811-1},
we first notice that $I_{T,0}\subset I_T.$
Let
 $\{x_n\}\in  I_T.$
Fix $a\in  A_+^{\bf 1}\setminus \{0\}.$ 
{{Choose $0<\eta_1<\|a\|/2.$
Then $(a-\eta_1)_+\in {\rm Ped}(A)_+\setminus \{0\}.$}}  
Since $T_e$ is compact, 
$A$ is simple  and $\tau\mapsto \tau(a-\eta_1)_+$ is continuous,
we have 
$$\sigma_0:
= \inf\{\tau((a-\eta_1)_+): \tau\in T_e\}>0.$$
Fix any $\ep\in (0,\eta_1).$
Then 
\beq\label{Ptace-e-1}
f_{\ep/2}(x_n^*x_n)\le (\frac{16}{\ep})(x_n^*x_n-\ep/8)_+\rforal n\ge 1.
\eneq
There exists $N\ge 1$ such that, for all $n\ge N,$
$\tau((x_n^*x_n-\ep/8)_+)<\frac{\ep\cdot \sigma_0}{16}$  for all $\tau\in T_e.$
By \eqref{Ptace-e-1},  we have
\beq
d_\tau(f_\ep(x_n^*x_n))\le \tau(f_{\ep/2}(x_n^*x_n))<\sigma_0\le d_\tau(a) \rforal \tau\in T_e\andeqn \rforal n\ge N.
\eneq
Since $A$ has strict comparison (see 
Definition \ref{Dqtr}), for all $n\ge N,$
$f_\ep(x_n^*x_n)\lesssim a.$ 
Thus $\{x_n\}\in N_{cu}(A).$
It follows that   $I_T \subset N_{cu}(A).$ {{Since $N_{cu}(A)$ is closed,   we conclude $\overline{I_{T,0}}\subset N_{cu}(A).$}}

Now let $\{x_n\}\in N_{cu}(A)$ and $\ep>0.$  We may assume that $\|\{x_n\}\|\le 1.$
Fix $\ep_1\in(0,\ep).$
For  any $\eta>0,$ 
since $A$ is simple, infinite-dimensional, and non-elementary, 
$\Her_A(e)$ is also simple,  infinite-dimensional,  and non-elementary. 
It follows from 
\cite[Lemma 4.3]{FL} (see also \cite[Lemma 2.4]{Lin1991}) 
that
there is $a_\eta\in \Her_A(e)_+\setminus \{0\}$
such that $d_\tau(a_\eta)<\eta$ for all $\tau\in T_e.$ 
 There exists an integer $N_1\ge 1$ such that
\beq
(x_n^*x_n-\ep_1)_+\lesssim a_\eta\rforal n\ge N_1.
\eneq
It follows   that, for all $n\ge N_1,$
$$
\sup\{\tau((x_n^*x_n-\ep_1)_+): \tau\in 
T_e\}\le 
 \sup\{d_\tau((x_n^*x_n-\ep_1)_+): \tau\in 
 T_e\}
 \le
 \sup\{d_\tau(a_\eta): \tau\in 
 T_e\}
\leq \eta.
 $$
Therefore $\lim_{n\to\infty}\sup\{\tau((x_n^*x_n-\ep_1)_+):
\tau\in T_e\}=0.$ In other words,
$\{x_n\}\in I_T.$
This proves the first part of the proposition.

If, in addition, we assume  $A={\rm Ped}(A),$ 
then,  by Proposition 5.6 of \cite{EGLN},  
there are $M(e)>0$ and an integer $N(e)\ge 1$ such that, for any 
$b\in A_+^{\bf 1},$ there are $y_1, y_2,...,y_m\in A$ with $\|y_i\|\le M(e)$ and 
$m\le N(e)$ such that
\beq
\sum_{i=1}^m y_i^*ey_i=b.
\eneq
Let $\tau$ be a 2-quasi-trace on $A={\rm Ped}(A),$ 
which extends to a 2-quasi-trace on $M_m(A).$ 
Let $Z:=(z_{i,j})_{m\times m},$ where, {{for each $i,$}}  $z_{i,1}=y_i$ and $z_{i,j}=0$ for $1<j\le m$ and
{{${\bar e}_m=\diag(e,e,...,e).$}}
We  then estimate
\beq
\tau(b)&=&\tau(\sum_{i=1}^m y_i^*ey_i) = 
\tau(Z^*{\bar {e}}_mZ)=\tau(({\bar e}_m)^{1/2}ZZ^*({\bar e}_m)^{1/2})\\
&\le& \|ZZ^*\|\tau({\bar e}_m)=
 \|Z\|^2 \cdot m\tau(e)\le  N(e)^3 M(e)^2  \tau(e).
\eneq
It follows 
that
\beq
\Delta:=\sup\{\|\tau\|: \tau\in T_e 
\}\le N(e)^3M(e)^2.
\eneq
 {{Let $\{x_n\}\in I_{T}$ 
 and let $\ep,\dt>0.$}} 
 Choose $\ep_1:=\ep/N(e)^3M(e)^2.$
 {{Then there is $N \in\mathbb{N}$
 such that}} 
 \beq
 {{\sup\{\tau((x_n^*x_n-\ep_1)_+):\tau\in T_e\}  <  \dt
 \quad \mbox{ for all } n\geq N.}}
 \eneq
{{Let $n\geq N$ and $\tau\in T_e.$}} 
Consider {{the}} 
commutative \SCA\, $B$  of {{$A$}}
generated by $x_n^*x_n.$ 
Then $\tau$ extends {{to}}
a positive linear functional on $\wtd B.$ 
Then,  
\beq
\tau(x_n^*x_n) &=&\tau((x_n^*x_n-{{\ep_1}}
)_+)-\tau((x_n^*x_n-{{\ep_1}}
)_-)+\tau({{\ep_1}}
)\\
&\le & 
{{\tau((x_n^*x_n-\ep_1 
)_+)+\ep_1\|\tau\|
< \dt+\ep.}}
\eneq
This implies 
\beq
\lim_{n\to\infty}\sup\{\tau(x_n^*x_n): \tau\in 
T_e\}=0.
\eneq
It follows that $\{x_n\}\in I_{T,0}.$    
Therefore, we have shown $I_T\subset I_{T,0}.$

In conclusion, when $A={\rm Ped}(A),$ we have
$I_{T,0}=N_{cu}(A)=I_{T}.$
\vspace{-0.1in}
\end{proof}

\begin{rem}\label{R39}
In Proposition \ref{equiv-ideal-F812-1}, if $A={\rm Ped}(A),$ 
for any $e\in A_+\setminus \{0\},$
\beq
I_{T,0}=\{\{x_n\}\in l^\infty(A): \lim_{n\to\infty} \sup\{\|x_n\|_{2,\tau}: \tau\in T_e\}=0\}
= N_{cu}(A),
\eneq
where  
$\|x_n\|_{2, \tau}= \tau(x_n^*x_n)^{1/2}.$  
{{Note that,
$I_{T,0}$  and $I_T$ are  independent of the choice of $e$
in ${\rm Ped}(A)_+\setminus \{0\}.$}}  

However, $N_{cu}{{(A)}}\not=I_{T,0}$ in general.  To see this, 
let $B$ be a unital separable simple 
\CA\  which has a nontrivial 
2-quasi-trace
and $A=B\otimes {\cal K}.$
Let $\{e_{i,j}\}$ be a system of matrix units for ${\cal K}.$ 
Fix $e\in  {\rm Ped}(A)_+\setminus \{0\}.$
Note that $1/2<\sum_{i=n}^{2n} (1/i)<1.$ 
Define $y_n=\sum_{i=n}^{2n} (1/\sqrt{i})(1_B\otimes e_{i,i})\in A,$ $n\in \N.$  Note $y_n^*y_n\in {\rm Ped}(A)_+$
and $\|y_n\|\le 1$. 
For any $\ep>0,$ let $m\in \N$ such that $1/m<\ep.$
Then $(y_n^*y_n-\ep)_+=0$ for all $n>m.$ Therefore $\{y_n\}\in N_{cu}(A).$
Also,
 $\tau(y_n^*y_n)>(1/2)\tau(1_B\otimes e_{1,1})
 $ for all $\tau\in T_e$ 
 (but $\tau(y_n^*y_n)<\tau(1_B\otimes e_{1,1})$
for each $n$).  So $\{y_n\}\not\in I_{T,0}.$
\end{rem}

Recall Definition \ref{seq-def-F812-1}.
We have the following version of central surjectivity
(c.f. \cite[Proposition 4.5(iii) and Proposition 4.6]{KR14}, 
see also
\cite[Theorem 3.1]{MS2014}). 
Note that 
the following proposition is 
related to 
the so-called $\sigma$-ideal 
(\cite[Definition 1.5, Proposition 1.6]{Kir06}). 

\begin{prop}
\label{central-surj}
For a non-elementary separable simple
\CA\ $A,$
the canonical {{maps}} ${{\pi'}}: \pi_\infty(A)' \to \pi_{cu}(A)'$ 
{{and $\pi^\bot: \pi_\infty(A)^\bot \to \pi_{cu}(A)^\bot$ are}}
surjective. 
\end{prop}

\begin{proof}
Let $\{d_n\}\subset A_+^{\bf 1}$ be a sequence of positive contractive elements with 
$\|d_n\|=1$ for all $n\in\mathbb{N}$
such that for any $a\in A_+\backslash\{0\},$ there exists $N\in\mathbb{N}$
satisfying $d_n\lesssim a$ for all $n\geq N$
(see \cite[Lemma 4.3]{FL}).
Let $x=\{x_n\}\in l^\infty(A)$ such that 
$\pi_{cu}(x)\in \pi_{cu}(A)'$
and $g=\{g_n\}\in  l^\infty(A)$ such that $\pi_{cu}(g)\in \pi_{cu}(A)^\perp.$
Let $F_1\subset F_2\subset ...\subset A$ be a sequence of finite subsets
with $\overline{\cup_{m}F_m}=A.$
Let $m\in\mathbb{N},$ 
then $\pi_{cu}(x)\pi_{cu}(y)-\pi_{cu}(y)\pi_{cu}(x)=0$ 
{{and $\pi_{cu}(g)\pi_{cu}(y)=0=\pi_{cu}(y)\pi_{cu}(g)$}}
for all $y\in F_m.$
i.e. $xy-yx, gy, yg \in N_{cu}(A).$ 
By the existence of quasi-central approximate identity, 
there exists $e^{(m)}=\{e^{(m)}_n\}_{n=1}^\infty\in N_{cu}(A)_+^{\bf 1}$ such that 
\beq
\|(1-e^{(m)})(xy-yx)(1-e^{(m)})\|<1/3m
\mbox{ and }
\|e^{(m)}y-ye^{(m)}\|<1/3m
\mbox{ for all } y\in F_m, 
\eneq
which implies  
\beq
\label{Censurj-1}
\|(1-e^{(m)})x(1-e^{(m)})y-y(1-e^{(m)})x(1-e^{(m)})\|<1/m
\mbox{ for all }
y\in F_m.
\eneq
Similarly, we also assume that
\beq
\label{Censurj-1+}
&&\hspace{-0.7in}\|(1-e^{(m)})g(1-e^{(m)})y\|<1/m\andeqn \|y(1-e^{(m)})g(1-e^{(m)})\|<1/m
\mbox{ for all }
y\in F_m.
\eneq
Let  
$z^{(m)}:=x-(1-e^{(m)})x(1-e^{(m)})=e^{(m)}x+xe^{(m)}-e^{(m)}xe^{(m)}
\in N_{cu}(A)$
and let $\zeta^{(m)}:=g-(1-e^{(m)})g(1-e^{(m)})=e^{(m)}g+ge^{(m)}-e^{(m)}ge^{(m)}\in N_{cu}(A).$
Write $z^{(m)}=\{z^{(m)}_k\}_{k=1}^\infty$ and 
$\zeta^{(m)}=\{\zeta^{(m)}_k\}{_{k=1}^\infty}.$
{{Then,}} 
for any $m\in\mathbb{N},$ there is $K(m)\in \mathbb{N}$
such that 
\beq
\label{Censurj-5}
f_{1/m}(z^{(m)*}_kz^{(m)}_k)\lesssim d_m\andeqn
f_{1/m}(\zeta^{(m)*}_k\zeta^{(m)}_k)\lesssim d_m  \mbox{ for all }k\geq K(m).
\eneq
We may assume that $K(m+1)>K(m)>0$ 
for all $m\in\mathbb{N}.$ 
For each $k\geq K(1),$ define $m_k:=\max\{m\in\mathbb{N}: K(m)\leq k\}<\infty.$
Note that 
\beq
\label{Censurj-6}
K(m_k)\leq k.
\eneq
For $k<K(1),$ {{define}}  $w_k=0=v_k.$
For $k\geq K(1),$ define 
\beq
w_k:=z^{(m_k)}_k\andeqn v_k:=\zeta^{(m_k)}_{k}.
\eneq
For any $\ep>0$ {{and}} any $a\in A_+\backslash\{0\},$
let $r_1\in \mathbb{N}$ such that $d_m\lesssim a$ for all $m\geq r_1.$
Then, for any $k\geq K( \max\{[1/\ep]+1,r_1\}),$ 
we have $m_k\geq  \max\{1/\ep,r_1\},$ and 
by \eqref{Censurj-6} and \eqref{Censurj-5}, we have 
\beq
&&f_{\ep}(w_k^*w_k)
=
f_{\ep}(z^{(m_k)*}_kz^{(m_k)}_k)
\lesssim
f_{1/m_k}(z^{(m_k)*}_kz^{(m_k)}_k)
\lesssim d_{m_k}\lesssim a\andeqn\\
&&f_{\ep}(v_k^*v_k)
=
f_{\ep}(\zeta^{(m_k)*}_k\zeta^{(m_k)}_k)
\lesssim
f_{1/m_k}(\zeta^{(m_k)*}_k\zeta^{(m_k)}_k)
\lesssim d_{m_k}\lesssim a,
\eneq
which {{shows}} 
$\{w_k\},\, \{v_k\}\in N_{cu}(A).$ 

Now define 
\beq
&&\bar x_k:=x_k-w_k
=x_k-z^{(m_k)}_k
=(1-e^{(m_k)}_k)x_k(1-e^{(m_k)}_k)\andeqn\\
&&\bar g_k:=g_k-v_k
=g_k-\zeta^{(m_k)}_k
=(1-e^{(m_k)}_k)g_k(1-e^{(m_k)}_k).
\eneq 
Since $\{w_k\}, \{v_k\}\in N_{cu}(A),$ we have
\beq
\label{Censurj-2}
\nonumber
&&\pi(\pi_\infty(\{\bar x_k\}))-\pi_{cu}(x)
=\pi_{cu}(\{\bar x_k-x_k\})
=
-\pi_{cu}(\{w_k\})=0\andeqn\\
&&\pi(\pi_\infty(\{\bar g_k\}))-\pi_{cu}(g)
=\pi_{cu}(\{\bar g_k-g_k\})
=
-\pi_{cu}(\{v_k\})=0.
\eneq

Fix $r\in \mathbb{N}$ and $y\in F_{r}.$ Let $\dt>0.$ 
{{Then,  for}}
any $k \geq K(\max\{r,[1/\dt]+1\}),$ 
we have $m_k\geq r,$ 
$y\in F_{m_k},$ and $1/m_k\leq \dt.$   {{In particular,
\beq\label{Censurj-10}
\lim_{k\to\infty}1/m_k=0.
\eneq
}}
By \eqref{Censurj-1} and \eqref{Censurj-1+},  {{for $y\in F_r$ 
{{and}}
$k\geq K(\max\{r,[1/\dt]+1\}),$}} 
\beq
&&\hspace{-0.5in}\|\bar x_ky-y\bar x_k\|
=\|(1-e^{(m_k)}_k)x_k(1-e^{(m_k)}_k)y-y(1-e^{(m_k)}_k)x_k(1-e^{(m_k)}_k)\|
\leq 1/m_k,\\
&&\hspace{-0.5in}\|{{\bar g_k}}
y\|
=\|(1-e^{(m_k)}_k)
{{g_k}}
(1-e^{(m_k)}_k)y\|
\leq 1/m_k,\andeqn\\
&&\hspace{-0.5in}\|y 
{{\bar g_k}}
\|
=\|y(1-e^{(m_k)}_k)
{{g_k}}
(1-e^{(m_k)}_k)\|
\leq 1/m_k.
%
\eneq
{{Combining with \eqref{Censurj-10}, this   implies that, for each $y\in F_r,$ 
\beq
&&\|\pi_\infty(\{\bar x_k\}){{\pi_\infty(y)-\pi_\infty(y)}}\pi_\infty(\{\bar x_k\})\|=0\andeqn
\\
&&\|\pi_\infty(\{\bar g_k\}){{\pi_\infty(y)}}\|=0=\|{{\pi_\infty(y)}}\pi_\infty(\{\bar g_k\})\|.
\eneq}}
Since $\overline{\cup_{r}F_r}=A,$ we have 
\beq
\label{Censurj-3}
\pi_\infty(\{\bar x_k\})\in \pi_\infty(A)'\andeqn \pi_\infty(\{\bar g_k\})\in 
{{\pi_\infty(A)^\bot.}}
\eneq
Then \eqref{Censurj-2} and \eqref{Censurj-3} show that 
$\pi'$ and $\pi^\bot$ are surjective.
\end{proof}


\section{Tracial approximate divisibility}

\begin{df}\label{D-TAD}
Let $A$ be a simple \CA.  We say that $A$ has property (TAD) if the following holds:
 for any $\ep>0,$  
{{any}} 
finite subset ${\cal F}\subset A,$
 any
$s\in A_+\setminus \{0\},$  and any integer $n\ge 1,$
there are  $\theta\in A_+^{\bf 1}$  and  a  \SCA\, $D\otimes M_n\subset A$
such that

(i)  $\theta x\approx_{\ep} x\theta$
for all $x\in {\cal F},$

(ii) ${{(1-\theta)}}x\in_{\ep} D\otimes 1_n$ for all $x\in {\cal F},$
 and

(iii) $\theta\lesssim s.$

\end{df}

\begin{rem}\label{R1}
(1) It is straightforward to show that if $A$ has (TAD), we may further require that

(iv) ${{(1-\theta)x\approx_\ep}}
(1-\theta)^{1/2} x(1-\theta)^{1/2}\in_{\ep} D\otimes 1_n\rforal x\in {\cal F},$
and 

(v) $x\approx_{\ep} \theta^{1/2}x\theta^{1/2}+(1-\theta)^{1/2}x(1-\theta)^{1/2}$ for all $x\in {\cal F}.$

{{(2)}}
It is also 
easy to see that if $A{{\ \neq \mathbb{C}}}$ has property (TAD), then for any integer $n\ge 1 ,$
$M_n(A)$ has the property (TAD) as well.   

{{(3)}} If $A=\overline{\cup_{n=1}^\infty A_n},$
where each $A_n$ has the property (TAD), then $A$ has property (TAD). 
To see this, let $\ep>0,$ ${\cal F}\subset A$ be a finite subset and 
$s\in A_+\setminus \{0\}.$  Choose $n\ge 1$ such 
that, $x{{\ \in\ }}_{\ep/4} A_n$ for all $x\in {\cal F}$ and $a\in (A_n^{\bf 1})_+$
such that $s_1:={{(a-\ep\|a\|/4)_+}}\lesssim s.$  Using the assumption that $A_n$ has property (TAD), 
one concludes that $A$ property (TAD).   As a consequence, 
if $A\not=\C$  and has property (TAD), 
then $A\otimes {\cal K}$ has property (TAD).

\end{rem}

Recall the definition of {{tracial approximate divisibility}} from 
\cite{FLII}: 
\begin{df}\label{Dtrdivisible}
\cite[Definition 5.2]{FLII}
Let $A$ be a  {{simple}} \CA. $A$ is said to  be tracially approximately divisible,
if, for any $\ep>0,$ {{any}} finite subset ${\cal F}\subset A,$
any element $e_F\in A_+^{\bf 1}$ with
$e_Fx\approx_{\ep/4}x\approx_{\ep/4}xe_F$ {{for all $x\in {\cal F},$}}
 any
$s\in A_+\setminus \{0\},$  and any integer $n\ge 1,$
there are  $\theta\in A_+^{\bf 1},$   a \SCA\, $D\otimes M_n\subset A$  and  a c.p.c.~map
$\bt: A\to A$
such that

(1)  $x\approx_\ep x_1+\bt(x)$ {{for all $x\in{\cal F}$}},
where $\|x_1\|\le \|x\|,$ $x_1\in {\rm Her}(\theta),$ 

(2) $\bt(x)\in_\ep D\otimes 1_n$ and $e_F\bt(x)\approx_{\ep} \bt(x)\approx_\ep \bt(x)e_F$  for all $x\in {\cal F},$  and

(3) $\theta \lesssim s. $
\end{df}

By Proposition 5.3 of \cite{FLII},  for a simple \CA\ $A,$ if $A$ has property (TAD), then 
$A$ is tracially approximately divisible.

\hspace{0.2in}

In \cite{HO} (Definition 2.1),  a  unital \CA\ $A$ which is not $\C$  is called tracially ${\cal Z}$-absorbing, 
if  for any finite subset ${\cal F}\subset A,$  {{any}} $\ep>0,$ any 
$s\in A_+\backslash\{0\},$ 
and any integer $n\ge 1,$   there is an order zero 
c.p.c.~map
$\phi: M_n\to A$  such that 
the following condition hold:

(i) $\phi(g)x\approx_{\ep} x\phi(g) \rforal x\in {\cal F}\tand g\in M_n^{\bf 1},$ and 

(ii)  $1-\phi(1_n)\lesssim s.$ 

\ 

We state a non-unital 
variation of this notion (taken from \cite{AGJP17}, see also \cite[Definition 6.6]{FG17} and \cite[Definition 2.1]{CLS}) as follows. 

\begin{df}\label{TAD}
[c.f. \cite{AGJP17}]
Let $A$ be a simple \CA.    We say that 
$A$ has property (TAD-2) if 
the following holds:
for any $\ep>0,$  
any 
finite subset ${\cal F}\subset A,$
any $e_F\in A_+^{\bf 1}$ with $e_Fx\approx_{\ep}x\approx_{\ep}xe_F,$
 any
$s\in A_+\setminus \{0\},$  
any integer $n\ge 1,$
and any finite subset ${\cal G}\subset C_0((0,1])\otimes M_n,$ 
there is a \hm\, $\phi: C_0((0,1])\otimes M_n\to A$ 
such that
(recall that $\iota$ is the identity function  on $(0,1]$)

(1)   $\phi(g)x\approx_{\ep} x\phi(g)$
for all $x\in {\cal F}$ and $g\in {\cal G},$ 
and

(2) $((e_F-e_F^{1/2}\phi({{\iota\otimes 1_n}})e_F^{1/2})-\ep)_+\lesssim s.$
\end{df}

\begin{rem}\label{Rafter44}
When $A$ has a unit $1_A$, 
then (2) in the above definition 
is equivalent to 
$((1_A-\phi(\iota\otimes 1_n))-\ep)_+\lesssim s.$

By choosing $e_F=1_A,$ (2) becomes $((1_A-\phi(\iota\otimes 1_n))-\ep)_+\lesssim s.$ 
Conversely, let $1/2>\ep>0,$ ${\cal F},$ $e_F,$ $s,$ $n$ and ${\cal G}$
be given.
Define $\iota_1\in C_0((0,1])_+$ by $\iota_1(t)=1$ if $t\in [1-\ep,1],$ $\iota_1(0)=0$ and $\iota_1(t)$ is linear on $[0,1-\ep).$ 
Define a \hm\,  $\af: C_0((0,1])\otimes M_n\to C_0((0,1])\otimes M_n)$ by $\af(f(\iota)\otimes e_{i,j})=
f(\iota_1)\otimes e_{i,j}$ for all $f\in C_0((0,1]).$
Put ${\cal G}_1=\{\af(g): g\in {\cal G}\}.$ Suppose that $\phi$ is as in the definition associated with $\ep,$
${\cal F},$ $s,$ $n$ and ${\cal G}_1.$ 
Define $\psi: C_0((0,1])\otimes M_n\to A$ by $\phi\circ \af.$
Then, we have

 (i) $\psi(g) x=\phi\circ \af(g)x\approx_{\ep} x \psi(g)$ for all $g\in {\cal G}.$ 
Moreover, 
(ii)
\beq
1_A-\psi(\iota\otimes 1_n)=1_A-\phi\circ \af(\iota\otimes 1_n)
=\frac{1}{1-\ep}(1_A-\phi(\iota\otimes 1_n)-\ep)_+\lesssim s.
\eneq
In other words, in the case that $A$ is unital, the property (TAD-2) is equivalent to the 
property of tracially ${\cal Z}$-absorbing in the sense of Definition 2.1 of \cite{HO}.

{{Let $\psi': M_n\to A$ be the c.p.c.~order zero map defined by
$\psi'(e_{i,j})=\phi(\iota\otimes e_{i,j})$ ($1\le i, j\le 1$). Since the unit ball of $M_n$ is compact, with a large
${\cal G},$ (1) is equivalent to that $\|[\psi'(g),\, x]\|<\ep$ for all 
$x\in{\cal F}$ and 
$g\in M_n^{\bf 1}.$}}
\end{rem}

The following Lemma \ref{perp-ideal-F814-1}
and Corollary \ref{perp-ideal-F814-2}
are taken from \cite{Kir06}. 
We include proofs here for the reader's convenience.

\begin{lem}
\label{perp-ideal-F814-1}
{{\rm (c.f. \cite[Proposition 1.9 (3)]{Kir06})}}
{{Let $A$ be a \CA, $B\subset A$ {{be}} a subset, 
$B':=\{a\in A:ab=ba\mbox{ for all } b\in B\},$
$B^\bot:=\{a\in A:ab=0=ba\mbox{ for all } b\in B\},$ 
then $B^\bot$ is a closed two-sided of $B'.$ 
Let $\pi:B'\to B'/B^\bot$ be the quotient map. 
Suppose that $e\in B'$ satisfies  $eb=be=b$ for all $b\in B,$  
then $B'/B^\bot$ is unital and  $\pi(e)$ is the unit.}}
\end{lem}
\begin{proof}
It is straightforward to see that $B^\bot$ is a closed 
two-sided ideal of $B'.$ 
If $x\in B'$, then for any $b\in B,$ 
$(ex-x)b=b(ex-x)=bex-bx=bx-bx=0.$
Also, $(xe-x)b=xeb-xb=xb-xb=0.$ 
Then $xe-x,ex-x\in B^\bot.$ 
Thus $\pi(e)\pi(x)=\pi(x)=\pi(x)\pi(e).$ 
This completes the proof.
\end{proof}

\begin{cor}
\label{perp-ideal-F814-2}
{{{\rm (c.f. \cite[Proposition 1.9 (3)]{Kir06})}}}
Let $A$ be a $\sigma$-unital \CA\ without one-dimensional hereditary 
$C^*$-subalgebras.
Then both $\pi_\infty(A)'/\pi_\infty(A)^\bot$ and 
$\pi_{cu}(A)'/\pi_{cu}(A)^\bot$ are unital.
\end{cor}
\begin{proof}
Let $e\in A_+^{\bf 1}$ be a strictly positive element. 
Set $e_0=\{e^{1/n}\}\in l^\infty(A).$ 
Then $\pi_\infty(e_0)$ (resp. $\pi_{cu}(e_0)$) 
is a local unit of $\pi_\infty(A)$ (resp. $\pi_{cu}(A)$). 
Hence, by
Lemma \ref{perp-ideal-F814-1}, the lemma holds.
\end{proof}

\begin{df}
Let $A$ be a 
$\sigma$-unital non-elementary 
simple \CA. 
We say that $A$ has property (TAD-3), 
if, 
for any $n\in \N,$   there is  a unital \hm\ 
$\phi: M_n\to  \pi_{cu}({{A}})'/\pi_{cu}({{A}})^\bot.$
\end{df}

\begin{lem} \label{TtadtoTAD2}
Let $A$ be a non-elementary  
separable 
simple 
\CA.

(1) If $A$ 
 is 
 tracially approximately divisible, then 
$A$ has 
property (TAD-3).

(2) If  $A$  has property (TAD-3), then 
$A$ has the property (TAD-2).

\end{lem}

\begin{proof}
{\em Proof of (1).}
Fix $N\in\mathbb{N}.$
By \cite[Lemma 4.3]{FL}, 
there exists a sequence 
$\{s_n\}\subset A_+^{\bf 1}\backslash  
\{0\}$ such that, for any $a\in {{A}}_+\setminus \{0\},$ there exists $n_0\ge 1$ such that 
$s_n\lesssim a$ for all $n\geq n_0.$
Choose $0<\ep_n<\ep$ such that $\sum_{n=1}^\infty \ep_n<1$ and 
an increasing sequence of finite subsets ${\cal F}_n\subset {{A}}^{\bf 1}$ such that 
and $\cup_{n=1}^\infty {\cal F}_n$ is dense in ${{A^{\bf 1}}}.$ 

If $A$ is tracially approximately divisible, then
there exist a sequence of \SCA s $D_n\otimes M_N\subset A,$
a sequence of c.p.c.~maps $\bt_n: A\to A,$ 
and
a sequence of positive elements $\theta_n\in A^{\bf 1}$
satisfying the following:
for any $x\in {\cal F}_n,$ 
there is 
$x^{(n)}\in {\rm Her}(\theta_n)^{\bf 1}$ such that 

(i) $x\approx_{\ep_n} x^{(n)}+\bt_n(x),$

(ii) $\bt_n(x)\in_{\ep_n}  D_n\otimes 1_N,$ and 

(iii)  $\theta_n\lesssim s_n.$

In particular, 
we have that $\{\theta_n\}\in N_{cu}(A).$ 

For convenience, we 
put $J=\pi_{cu}(A)^\perp.$ {{By Lemma \ref{perp-ideal-F814-1},
$J$ is an ideal of }$\pi_{cu}(A)'.$}
Denote by $\pi_{cu,J}: \pi_{cu}(A)'\to \pi_{cu}(A)'/J$ 
the quotient map.

For each $x\in {\cal F}_k,$   we have 
$\{x^{(n)}\}\in N_{cu}(A).$ 
It follows from (i)
that, for each $k$ and each $x\in {\cal F}_k,$
\beq\label{TAD2-c-1}
\pi_{cu}(\{\bt_n(x)\})=\pi_{cu}(x).
\eneq
Choose $d_n\in {D_n^{\bf 1}}_+$ such that, for any $x\in {\cal F}_k,$ 
\beq
&&
(d_n\otimes e_{i,j})\bt_n(x)
\approx_{2\ep_n}
\bt_n(x)(d_n\otimes e_{i,j})
\andeqn\\
&&
(d_n\otimes 1_N)\bt_n(x)
\approx_{2\ep_n} 
\bt_n(x)
\approx_{2\ep_n} 
\bt_n(x)(d_n\otimes 1_N).
\eneq
Since $\cup_{n=1}^\infty {\cal F}_n$ is dense in ${{A^{\bf 1}}},$ combining this with \eqref{TAD2-c-1},  we obtain that 
\beq
&&\pi_{cu}(\{d_n\otimes e_{i,j}\})\pi_{cu}(x)=\pi_{cu}(x) \pi_{cu}(\{d_n\otimes e_{i,j}\})\andeqn\\\label{TAD2-c-2}
&&\hspace{-0.3in}
\pi_{cu}(\{d_n\otimes 1_N\})\pi_{cu}(x)
=\pi_{cu}(x)=\pi_{cu}(x) \pi_{cu}(\{d_n\otimes 1_N\})
\rforal x\in {{A.}} 
\label{TAD2-F812-1}
\eneq
In other words, $\{
\pi_{cu}(\{d_n\otimes e_{i,j}\}):1\leq i,j\leq N\}\subset \pi_{cu}(A)',$
and $\pi_{cu}(\{d_n\otimes 1_N\})$ 
is a local unit of $\pi_{cu}(A).$

{{By \eqref{TAD2-F812-1} 
and Lemma \ref{perp-ideal-F814-1}, $\pi_{cu,J}(\{d_n\otimes 1_N\})$ is the unit 
of $\pi_{cu}(A)'/J.$ 
Then the map $\phi: M_N\to   \pi_{cu}(A)'/J$  defined 
by $\phi(e_{i,j})=\pi_{cu,J}\circ\pi_{cu}(\{d_n\otimes e_{i,j}\}),$ $1\le i,j\le N,$ 
is a}}
 unital 
 homomorphism. 
Since, for every $N\in\mathbb{N},$ there is a unital embedding $\phi: M_N\to
\pi_{cu}(A)'/\pi_{cu}(A)^\bot,$
$A$ has property (TAD-3).


{{\em Proof of (2).}}
Let $\ep>0$ and any finite subset ${\cal F} \subset A,$  any $s \in A_+\backslash \{0\},$  and any integer $N\ge 1$ be given.
Choose $e_F\in A_+^{\bf 1}$ such that 
\beq
e_Fx\approx_{\ep/32} xe_F\approx_{\ep/32} x\rforal x\in {\cal F}.
\eneq
Recall that $J=\pi_{cu}(A)^\perp.$ 
Since $A$ has (TAD-3),
there exists a
{{unital embedding}} 
$\phi: M_{{N}}\to \pi_{cu}(A)'/J$
such that
\beq\label{TAD-e-1}
\phi(1_{{N}})\pi_{cu,J}(a)=\pi_{cu,J}(a)=\pi_{cu, J}(a)\phi(1_N)\rforal a\in A.
\eneq
Define a \hm\, $\Phi: C_0((0,1])\otimes M_N
\to  \pi_{cu}(A)'/J$ 
by $\Phi(f\otimes e_{i,j})=\phi(f(1)\otimes e_{i,j})$
for all $f\in C_0((0,1]),$ $1\le i,j\le N.$

By {{Proposition \ref{central-surj},}} 
${{\pi}}(\pi_\infty(A)')=\pi_{cu}(A)'.$  Since 
$C_0((0,1])\otimes M_N$ is projective, 
there is a \hm\, $\Psi: C_0((0,1])\otimes M_N
\to \pi_\infty(A)'$ 
such that ${{\pi}}\circ \Psi=\Phi.$
We may write $\Psi=\{\psi_n\},$ 
where $\psi_n: C_0((0,1])\otimes M_N
\to A$ is a \hm. 
Thus, for any finite subset ${\cal G}\subset C_0((0,1])\otimes M_N,$ 
there exists $n_1\ge 1$ such that
\beq\label{TAD-e-0}
\psi_n(g)x\approx_{\ep} x\psi_n(g)\rforal x\in {\cal F,}
\mbox{ all }g\in{\cal G}
\mbox{ and all } n\geq n_1.
\eneq
Put $e_n:=\psi_n(\iota\otimes 1_N)$ 
and $\eta_n:=e_F^{1/2}-e_F^{1/4}e_ne_F^{1/4}.$
By \eqref{TAD-e-1},
\beq
\pi_{cu}(\{\eta_n\})=\pi_{cu}(\{e_F^{1/2}\}-\{e_F^{1/4}e_ne_F^{1/4}\})\in J.
\eneq
Thus, by Proposition \ref{central-surj}, 
there are $y=\{y_n\}\in N_{cu}(A)$ and 
$z=\{z_n\}\in l^\infty(A)$ 
with $\pi_\infty(z)\in \pi_\infty(A)^\bot,$
such that $\{\eta_n\}=y+z.$ 
So $\pi_\infty(aza)=0$  in $l^\infty(A)/c_0(A)$ for all $a\in A.$ 
It follows 
\beq
\pi_\infty(e_F-e_F^{1/2}
\{\psi_n(\iota\otimes 1_N)
\} e_F^{1/2})
=
\pi_\infty(e_F^{1/4}\{\eta_n\}e_F^{1/4})
=\pi_\infty(e_F^{1/4}ye_F^{1/4})
\in
\pi_\infty(N_{cu}(A)).
\eneq
In other words,
\beq
e_F-e_F^{1/2}\{\psi_n(\iota\otimes 1_N 
)\} e_F^{1/2}\in N_{cu}(A).
\eneq
It follows that there exists $n_{2}\ge {{n_1}}$ such that, for all $n\ge n_{{2}},$
\beq
(e_F-e_F^{1/2}\psi_n(\iota\otimes 1_N 
)e_F^{1/2}-\ep)_+\lesssim s.
\eneq
The lemma then  follows (see also \eqref{TAD-e-0}).
\end{proof}

\begin{prop}
\label{PTAD2toTAD}
Let $A$ be a 
simple \CA\,
 which has the property  (TAD-2).
Then $A$ has the property (TAD).
\end{prop}

\begin{proof}
Fix $\ep>0,$ a finite subset ${\cal F}\subset A$ and an integer $n\ge 1.$
We may assume that ${\cal F}\subset A^{\bf 1}.$ 
Choose $e_F\in A_+^{\bf 1}$ such that, {{for all $x\in {\cal F},$}}
\beq
\label{PTAD2toTAD-1}
e_F^{1/2}x\approx_{{{\ep/32n}}}
{{x}}
 \approx_{\ep/32n} xe_F^{1/2}\andeqn e_Fx\approx_{\ep/32n} x\approx_{\ep/32n} xe_F.
\eneq
Put ${\cal F}_1={\cal F}\cup\{e_F, e_F^{1/2}\}.$ 
Let $0<\eta<\ep.$  
Put
$${\cal G}=\{\iota\otimes 1_n, \iota^{1/2}\otimes 1_n, \iota\otimes e_{i,j}, \iota^{1/2}\otimes e_{i,j}, \iota\otimes 1_n: 1\le i,j\le n\}.$$
Let $e\in A_+\setminus \{0\}.$
Since $A$ has (TAD-2),
there is a \hm\, $\phi: C_0((0,1])\otimes M_n\to A$
such that   
\beq
\label{PTAD2toTAD-2}
\|[x,\, \phi(f)]\|<\eta/32n\rforal f\in {\cal G}\andeqn 
((e_F-e_F^{1/2}\phi(\iota)e_F^{1/2})-\ep/32n)_+\lesssim e.
\eneq
Put $d_1=\phi(\iota\otimes e_{1,1}).$ 
Define $D:=\overline{d_1Ad_1}.$ 
Put 
$$
C_0=\left\{\sum_{i,j=1}^n
\phi(\iota^{1/2}\otimes e_{i,1})
a
\phi(\iota^{1/2}\otimes e_{1,j}): a\in  A\right\}.
$$
Note that $C_0$ is a $*$-subalgebra {{of $A.$}} Put $C=\overline{C_0}.$
Define   a \hm\, $\Phi: C\to {{D\otimes M_n}}
$ by 
\beq
\Phi(\phi(\iota^{1/2}\otimes e_{i,1})a\phi(\iota^{1/2}\otimes e_{1,j}))=
d_1^{1/2}ad_1^{1/2}\otimes e_{i,j}
\eneq
for all $a\in A.$ It is easy to verify that $\Phi$ is an isomorphism. So $C\cong D\otimes M_n.$

Put $\theta=(e_F-e_F^{1/2}\phi(\iota\otimes 1_n)e_F^{1/2}-\ep/64n)_+.$
Then {{by \eqref{PTAD2toTAD-1} and \eqref{PTAD2toTAD-2},
we have the following (i)--(iii).}}
\beq
&&\hspace{-1in}{\rm (i)}\,\,\, \theta x \approx_{\ep/32n} (e_F-e_F^{1/2}\phi(\iota\otimes 1_n)e_F^{1/2})x
\approx_{4\ep/32n}x(e_F-e_F^{1/2}\phi(\iota\otimes 1_n)e_F^{1/2})\\
&&\hspace{-0.2in}\approx_{\ep/32n} x\theta\hspace{0.2in}
   \rforal x\in {\cal F}.
   \eneq
\beq
&&\hspace{-1.4in}{\rm (ii)}\,\,\,
(1-\theta)x\approx_{\ep/32n}
(1-(e_F-e_F^{1/2}\phi(\iota\otimes 1_n)e_F^{1/2}){{)}}x\\
&&\hspace{-0.4in}\approx_{\ep/32n} (e_F^{1/2}\phi(\iota\otimes 1_n)e_F^{1/2}{{)}}x
\approx_{2\ep/64n} \phi(\iota\otimes 1_n)x\\
&=&
\sum_{i=1}^n \phi(\iota\otimes e_{i,i})x
=\sum_{i=1}^n \phi(\iota^{1/2}\otimes e_{i,1})\phi(\iota^{1/2}\otimes e_{1,i})x\\
&&\hspace{-0.4in}\approx_{\eta/32} \sum_{i=1}^n\phi(\iota^{1/2}\otimes e_{i,1})x\phi(\iota^{1/2}\otimes e_{1,i})\in D\otimes 1_n.
\eneq

(iii) $\theta\lesssim e.$
\end{proof}

Now we can unify different variations of tracial approximate divisibility
for 
separable simple 
\CA s in the following theorem.

\begin{thm}\label{TTAD==}
Let $A$ be a 
non-elementary separable simple 
\CA. Then the following are equivalent.

(1) $A$ is tracially approximately divisible,

(2) $A$ has the property (TAD), 

(3) $A$ has the property (TAD-2), and,

(4) $A$ has the property (TAD-3).

\end{thm}

\begin{proof}
The implications $(1) \Rightarrow (4)$ and $(4)\Rightarrow (3)$ 
follow from {{Lemma}} \ref{TtadtoTAD2}. While $(3) \Rightarrow (2)$ follows from {{Proposition}} \ref{PTAD2toTAD}. Finally, $(2) \Rightarrow (1)$ follows from \cite[Proposition 5.3]{FLII}. 
\end{proof}

\begin{rem}
\label{rmk-F813-1}
Let $A$ be  a non-elementary separable simple \CA\, which satisfies one of four conditions in Theorem \ref{TTAD==}
and $B\subset A$ be a non-zero hereditary \SCA. Then, by  Theorem 5.5 of \cite{FLII} and  Theorem \ref{TTAD==} above, 
$B$ satisfies all conditions in Theorem \ref{TTAD==}.
\end{rem}

We would like to include the following statement.

\begin{prop}\label{Mcduff=TAD3}
Let $A$ be a unital non-elementary separable simple exact \CA\ with $T(A)\neq \emptyset.$
Assume $A$ has strict comparison. 
Then $A$ is uniformly McDuff in the sense of 
{\rm{\cite[Definition 4.2]{Gamma}}}
if and only if 
$A$ has the property (TAD-3).
\end{prop}

\begin{proof}
For  the ``only if'' part, 
let $n\in\mathbb{N}.$
Let $F_1\subset F_2\subset...\subset A$ be a sequence of 
finite subsets of $A$ with 
$\overline{\cup_{k\in \mathbb{N}} F_k}=A.$
Let $G_1\subset G_2\subset...\subset M_n$ be a sequence of 
finite subsets of $M_n$ with 
$\overline{\cup_{k\in \mathbb{N}} G_k}=M_n.$
Let $\omega$ be a free ultrafilter on $\mathbb{N}.$ 
Let $\|a\|_2:=\sup\{\tau(a^*a)^{1/2}:\tau\in T(A)\}.$
Let $I_\omega:=\{\{a_k\}\in l^\infty(A):\lim_{k\to\omega}
\|a_k\|_2=0\}.$
Let $A^\omega:=l^\infty(A)/I_\omega.$
Let $\pi^\omega:l^\infty(A)\to A^\omega$ 
be the quotient map. 
Since $A$ is uniformly McDuff, 
there is a unital embedding 
$\bar \phi: M_n\to A^\omega\cap \pi^\omega(A)'.$
By projectivity of $C_0((0,1])\otimes M_n,$
there is a c.p.c.~order zero map $\phi: M_n\to l^\infty(A)$
such that $\bar\phi=\pi^\omega\circ \phi.$
Let $\phi_m$ be the coordinate components of $\phi$
($m\in\mathbb{N}$). 
Then for any $k\in\mathbb{N},$
there is $m(k)\in\mathbb{N}$ 
such that 

(1) $\|\phi_{m(k)}(x)y-y\phi_{m(k)}(x)\|_2<1/k$ for all $x\in G_k,$
$y\in F_k,$

(2) $\|1_A-\phi_{m(k)}(1_n)\|_2<1/k.$ 

Then (1) shows that for all $x\in \cup_{k\in \mathbb{N}}G_k$
and all $ y\in \cup_{k\in \mathbb{N}}F_k,$
\beq
\label{F901-1}
\{\phi_{m(k)}(x)y-y\phi_{m(k)}(x)\}_{k=1}^{\infty}\in I_{T,0}.
\eneq
Note also that (2) implies  that 
$\{1_A-\phi_{m(k)}(1_n)\}_{k=1}^\infty\in I_{T,0}.$

Since $A$ is unital, non-elementary,  separable, simple,
and exact, we have 
 $I_{T,0}= N_{cu}(A)$
 (see  Proposition \ref{equiv-ideal-F812-1}).  
Define a c.p.c.~order zero map 
$\psi: M_n\to A_{cu}$ 
by $\psi(x)=\pi_{cu}(\{\phi_{m(k)}(x)\}_{k=1}^\infty)$ for all $x\in M_n.$
Then \eqref{F901-1} shows $\psi(x)\pi_{cu}(y)-\pi_{cu}(y)\psi(x)=0$
for all $x\in \cup_{k\in \mathbb{N}}G_k$ and all $y\in \cup_{k\in\mathbb{N}}F_k.$
Since 
$\overline{\cup_{k\in \mathbb{N}} F_k}=A$
and 
$\overline{\cup_{k\in \mathbb{N}} G_k}=M_n,$
it follows that 
$\psi(M_n)\subset \pi_{cu}(A)'.$

Since $\{1_A-\phi_{m(k)}(1_n)\}_{k=1}^\infty\in I_{T,0},$  we have 
$\pi_{cu}(1_A)=\psi(1_n).$
Then $\psi$ is actually a homomorphism. 
Hence $\psi$ gives a unital embedding 
from  $M_n$ to $\pi_{cu}(A)'.$
In other words, $A$ has the property (TAD-3). 

For the ``if'' part,   we only need to note that
$N_{cu}(A)=I_{T,0}\subset I_\omega.$
%
\end{proof}

\begin{rem}
We note that 
the implication from (TAD-2) 
to the uniform McDuff property is also proved 
in \cite[Proposition 4.6]{CLS}.

In view of Proposition \ref{Mcduff=TAD3}, one may make the following definition:
Let $A$ be a separable simple (exact) \CA\, with at least one nontrivial densely defined trace.
 Let $I_T$ be the closed ideal defined in  Proposition \ref{equiv-ideal-F812-1} and let
$\pi_{I_{T}}: l^\infty(A)\to  l^\infty(A)/I_{T}$ be the quotient map. We say that $A$ has the uniform McDuff property, if for each $n\in \N,$ there is a 
unital embedding $\phi: M_n\to  \pi_{I_{T}}(A)'/ \pi_{I_{T}}(A)^\perp.$ 
 
When $A$ also has strict comparison, $A$ is uniformly McDuff if and only if $A$ has property (TAD-3) (see Proposition \ref{equiv-ideal-F812-1} and Remark \ref{R39}).
\end{rem}

\section{Strict comparison and Cuntz semigroup}

\begin{prop}
\label{tech-prop-808f-1}
Let $A$ be a simple \CA\, with the property (TAD). 
Then,  for any integer $n\ge 1,$ any $\ep>0,$  
{{any}} 
finite subset ${\cal F}\subset A,$ and 
 any
$s\in A_+\setminus \{0\},$  
there are  $\theta\in A_+^{\bf 1}$  and  \SCA\, $D\otimes M_n\subset A$
such that

{\rm (i)}  $\theta x\approx_{\ep} x\theta$
for all $x\in {\cal F},$

{\rm (ii)} ${{(1-\theta)}}x\in_{\ep} D\otimes 1_n$ for all $x\in {\cal F},$

{\rm (iii)} $\theta\lesssim s, $ 

{\rm (iv)} $(1-\theta)^{1/2}x(1-\theta)^{1/2}\in _{\ep} D\otimes 1_n$ for all $x\in {\cal F},$

{\rm  (v)} $x\approx_{\ep} \theta^{1/2}x\theta^{1/2}+(1-\theta)^{1/2}x(1-\theta)^{1/2}$ for all $x\in {\cal F},$

%

{\rm (vi)} for any finite subset ${\cal G}\subset C_0((0,1]),$ there is 
$d\in D_+^{\bf 1}\backslash\{0\}$ 
such that, for all $x\in {\cal F},$
\beq
&&
(1-\theta)^{1/2}x(1-\theta)^{1/2}
\approx_{\ep/64n^2} (d\otimes 1_n)^{1/2}x(d\otimes 1_n)^{1/2}
%
\tand\\\label{PL-1-T-2}
&&x(f(d)\otimes e_{i,j})\approx_\ep (f(d)\otimes e_{i,j})x
\tforal 
f  \in {\cal G},\,
1\le i,j\le k,
\eneq

{\rm (vii)} if $x\in {\cal F}$ and $x\ge 0,$ 
we may choose $d$ such that 
\beq
x-x^{1/2}(d\otimes 1_n)x^{1/2}\approx_{\ep/4} x^{1/2}\theta x^{1/2}.
\eneq

\end{prop}

\begin{proof}
Fix $a\in A_+\backslash \{0\},$ 
$\ep\in (0,1),$ and a finite subset ${\cal F}_1\subset A.$ 
Without loss of generality, 
we may assume that 
for all $x\in {\cal F}_1,$
$\|x\|\leq 1.$
By a standard perturbation, we may assume that there is $e_F, e_A\in A_+^{\bf 1}$
satisfying 
$xe_F=e_Fx=x$ for all $x\in {\cal F}$ and $e_Ae_F=e_Fe_A=e_F.$ 
Put ${\cal F}={\cal F}_1\cup \{e_F, e_F^{1/2}, e_A\}.$ 

Let $m\in\mathbb{N}$ such that for any  $c\in A_+^{\bf 1},$ 
$c\approx_{\ep/4}c\cdot c^{1/m}.$
Let $\eta\in(0,\epsilon)$ {{be}} such that, for any element $z\in A^{\bf 1},$ 
 any $c\in A_+^{\bf 1},$ $\|zc-cz\|< 20\eta$ implies $\|zc^{1/m}-c^{1/m}z\|<\ep/4.$
 
{{By definition of (TAD),}} 
there are  $\theta\in A_+^{\bf 1}$  and  \SCA\, $D\otimes M_n\subset A$
such that

(1)  $\theta x\approx_{\eta} x\theta$
for all $x\in {\cal F},$

(2) ${{(1-\theta)}}x\in_{\eta} D\otimes 1_n$ for all $x\in {\cal F},$
 and

(3) $\theta\lesssim a.$

This implies  {{that}} (i), (ii), and (iii) in the {{proposition}} 
hold. On the other hand, (iv)  and (v) can be found in {{Remark}} \ref{R1}.

By (2), 
for each $x\in {\cal F},$ there exists ${{y\otimes 1_n}}\in D\otimes 1_n$
such that 
\beq
\label{PL-1-e-f1}
\|{{(1-\theta)x - y\otimes 1_n}}\|<\eta.
\eneq

Since $e_A\in {\cal F}$ 
{{and}} (iv) holds, we can 
choose $d\in D_+^{\bf 1}\backslash\{0\}$ 
such that 
\beq\label{PL-1-e-1}
(1-\theta)e_A\approx_\eta (1-\theta)^{1/2}e_A(1-\theta)^{1/2}
\approx_{2\eta} d\otimes 1_n.
\eneq
It follows that 
\beq
(y\otimes 1_n)(d\otimes 1_n)
&\approx_{4\eta}& 
(1-\theta)x(1-\theta) e_A
\\&\approx_{\eta}&
(1-\theta)xe_A(1-\theta)
=
(1-\theta)e_Ax(1-\theta)
\\&\approx_{\eta}&
(1-\theta)e_A(1-\theta)x
\approx_{5\eta}
(d\otimes 1_n)(y\otimes 1_n).
\eneq
Thus
$\|yd-dy\|<11\eta.$
Note that, 
by  \eqref{PL-1-e-1},
(1), \eqref{PL-1-e-f1} 
and the choice of $e_A,$ for all $x\in {\cal F}_1\cup\{e_F, e_F^{1/2}\},$ 
\beq\label{510-n2}
x (d\otimes 1_n) 
\approx_{3\eta}
x(1-\theta)e_A
\approx_\eta
(1-\theta)xe_A 
=
(1-\theta)x
\approx_\eta y\otimes 1_n.
\eneq
Similarly,
\beq\label{510n}
(d\otimes 1_n)x\approx_{3\eta} (1-\theta)e_A x=(1-\theta)x\approx_{4\eta} x(d\otimes 1_n).
\eneq
We compute that (recall $e_Ax=xe_A=x$ if $x\in {\cal F}_1$), for all $x\in{\cal F}_1,$
\beq\nonumber
&&\hspace{-0.7in}x(d\otimes e_{i,j})
\approx_{\epsilon/4} x (d\otimes 1_n) (d^{1/m}\otimes e_{i,j})
\approx_{4\eta} (1-\theta)x (d^{1/m}\otimes e_{i,j})\\\nonumber
&&\approx_{4\eta} (y\otimes 1_n)(d^{1/m}\otimes e_{i,j})=(yd^{1/m}\otimes e_{i,j})
\\\nonumber&&
\approx_{\ep/4} (d^{1/m}\otimes e_{i,j})(y\otimes 1_n)\\\nonumber
&&\approx_{4\eta} (d^{1/m}\otimes e_{i,j})(d\otimes 1_n)x
\approx_{\ep/4}(d\otimes e_{i,j})x.
\eneq

Thus, by choosing smaller $\eta$ (and $\ep$) if necessary, we conclude that  the second part 
(\eqref {PL-1-T-2}) of 
(vi) holds. 
The first part of (vi) 
follows from \eqref{PL-1-e-1} and a choice of small $\eta.$

To see (vii),   combing \eqref{510-n2} and \eqref{510n}, with sufficiently small $\eta,$ 
we have, if $x\in {\cal F}\cup\{e_F, e_F^{1/2}\}$ and $x\ge 0,$ 
\beq\nonumber
x^{1/2}(d\otimes 1_n)x^{1/2}\approx_{\ep/4} x^{1/2}(1-\theta)x^{1/2}, \,\,{\rm or}\,\,
x-x^{1/2}(d\otimes 1_n)x^{1/2}\approx_{\ep/4}x^{1/2}\theta x^{1/2}.
\eneq
\end{proof}

The following statement is {{already}} mentioned in 
Remark 5.8 of \cite{FLII}.

\begin{thm}[c.f. Theorem 3.3 of \cite{HO}]
\label{TAD-SC}
Let $A$ be a simple \CA\, which has property (TAD). Then $A$ has the strict comparison for positive elements 
(or $A$ is purely infinite). 

\end{thm}

\begin{proof} 
{{Following the original idea of R\o rdam, we will modify the argument in the proof of Lemma 3.2 of \cite{HO}.}}
Let us assume that $A$ is not elementary. 
Let  $a, b\in M_\infty(A)_+.$  Let us first
assume  that $0$ is not an isolated point  of 
${\rm sp}(b)\cup\{0\},$
and  $k\la a\ra \le k\la b\ra$ for some integer $k\ge 1,$
we wish to show that $a\lesssim b.$
\Wlog, we may  assume that $a, b\in M_N(A)_+$ for some large $N.$ Since $M_N(A)$ also 
has the property (TAD), we may assume  $a, b\in  A_+.$  We may further assume $\|a\|=\|b\|=1.$ 


Fix $\dt>0.$ 
By \cite[Proposition 2.4 (iv)]{Rordam-1992-UHF2}, 
we can choose $c=(c_{i,j})_{{{k}}\times {{k}}}\in M_k(A)$ and 
$\ep\in (0,\dt),$
such that
\beq\label{TAD=Comp-1}
c((b-\dt)_+\otimes 1_k)c^*=(a-\ep)_+\otimes 1_k.
\eneq
Since $0$ is not an isolated point of ${\rm sp}(b)\cup\{0\},$
there is a $f_0\in C_0((0,1])$ such that $f_0(t)=0$  for all $t\in  (\dt/2, 1]$ and 
$d:=f_0(b)\not=0.$  So $d\perp (b-\dt)_+.$  By replacing 
$c_{i,j}$ with $c_{i,j}q(b)$ for some $q\in C_0((0, 1])$ which vanishes in $(0, \dt/2]$  and $q(t)=1$
for all $t\in  [\dt,1],$ we may assume that 
\beq
\label{TAD=Comp-810-1}
c_{i,j}d=0,\quad i ,j=1,2,...,k.
\eneq
By \eqref{TAD=Comp-1}, we compute
\beq
\label{TAD=Comp-903}
\sum_{l=1}^k c_{i,l}((b-\dt)_+)c_{j,l}^{{*}}
=\begin{cases} (a-\ep)_+ & \text{if  $i=j;$}\\
                                                     0 & \text{if $i\not=j.$}\end{cases}
\eneq

{{Now let $\eta\in (0,\ep)$ be arbitrary.}} 
Put ${\cal F}_0=\{(a-\ep)_+, (b-\dt)_+\}\cup \{c_{i,j},{{c_{i,j}^*}}:1\le i,j\le k\}.$
{{Let $g \in C_0((0,1])_+^{\bf 1}$
such that $g(t)=1$ for $t\in [\eta,1].$
Let $M:=1+\max\{\|x\|:x\in{\cal F}\}.$}}

{{By Proposition \ref{tech-prop-808f-1},}}
there are  $\theta\in A_+^{\bf 1}$  and  \SCA\ $D\otimes M_k\subset A$
such that

(i)  $\theta x\approx_{{\eta}} x\theta$
for all $x\in {\cal F},$

(ii) ${{(1-\theta)}}x\in_{{\eta}} D\otimes 1_n$ for all $x\in {\cal F},$

(iii) $\theta\lesssim d, $ 

(iv) $(1-\theta)^{1/2}x(1-\theta)^{1/2}\in _{{\eta}} D\otimes 1_k$ for all $x\in {\cal F},$

(v) $x\approx_{{\eta}} \theta^{1/2}x\theta^{1/2}+(1-\theta)^{1/2}x(1-\theta)^{1/2}$ for all $x\in {\cal F},$

(vi) there is $e\in D_+^{\bf 1}$ such that, for all $x\in {\cal F},$
\beq
&&
(1-\theta)^{1/2}x(1-\theta)^{1/2}
\approx_{{\eta}}
(e\otimes 1_k)^{1/2}x(e\otimes 1_k)^{1/2}, 
%
\tand\\
&&{{\|x(g(e)\otimes e_{i,j})-(g(e)\otimes e_{i,j})x\|<{{\eta/(kM)^4}}
\quad(1\le i,j\le k).}}
\eneq
Put $a_1=(1-\theta)^{1/2}(a-\ep)_+(1-\theta)^{1/2}$ and $a_2=\theta^{1/2}(a-\ep)_+\theta^{1/2}.$
Then {{by (v),}}
\beq
\|(a-\ep)_+-(a_1+a_2)\|\leq {{\eta.}}
\eneq
{{Denote ${{\bar{c}}}:=\sum_{i,j=1}^k(e^{1/2}\otimes 1_k)(g(e)\otimes e_{i,j})c_{i,j}.$}}
We compute that 
(using \eqref{TAD=Comp-903})
\beq
&&
\quad
\bar{c} ((b-\dt)_+) \bar{c}^*
\\&&
=(e^{1/2}\otimes 1_k)\left(\sum_{i,j,l,m}^k
(g(e)\otimes e_{i,j})c_{i,j} 
((b-\dt)_+)
c^*_{l,m}(g(e)\otimes e_{m,l})\right) (e^{1/2}\otimes 1_k)
\\&&
\approx_{3\eta}(e^{1/2}\otimes 1_k)\left(\sum_{i,j,l,m}^k
(g(e)\otimes e_{i,j})(g(e)\otimes e_{m,l})c_{i,j} 
((b-\dt)_+)
c^*_{l,m}\right) (e^{1/2}\otimes 1_k)
\\&&
=(e^{1/2}\otimes 1_k) 
\left(\sum_{i,j,l}^k (g(e)^2\otimes e_{i,l})c_{i,j} ((b-\dt)_+)c^*_{l,j}\right)
(e^{1/2}\otimes 1_k)\\&&
=
(e^{1/2}\otimes 1_k) 
\left(\sum_{i,l=1}^k 
 (g(e)^2\otimes e_{i,l})
\left(\sum_{j=1}^k c_{i,j} ((b-\dt)_+)c^*_{l,j}\right)\right)
(e^{1/2}\otimes 1_k)\\
%
&&=
(e^{1/2}\otimes 1_k) 
\left(\sum_{i=1}^k 
 (g(e)^2\otimes e_{i,i})
(a-\ep)_+\right)
(e^{1/2}\otimes 1_k)
\\\nonumber
&&
=
(e^{1/2}\otimes 1_k)(g(e)^2\otimes 1_k)
 (a-\ep)_+ 
(e^{1/2}\otimes 1_k)\\\nonumber
&&
\approx_{\eta} (e^{1/2}\otimes 1_k)(a-\ep)_+
(e^{1/2}\otimes 1_k)\approx_{\eta} a_1.
\eneq
In other words,
$\|\bar{c}(b-\dt)_+\bar{c}^*-a_1\|\leq 4\eta.$

Since $a_2\lesssim \theta\lesssim d,$ there exists $c_0\in A$ such that
$\|c_0dc_0^*-a_2\|<\eta.$   Since $d\perp (b-\dt)_+,$ we may assume 
that $c_0((b-\dt)_+)=0.$ 
Now put $z=\bar{c}+c_0.$
Recall \eqref{TAD=Comp-810-1}, we have 
$$
\|z((b-\dt)_++
d)z^*-a\| 
=\|(\bar c(b-\dt)_+\bar c^* -a_1)+(c_0dc_0^*-a_2)+(a_1+a_2-a)\|<
5\eta+\ep.
$$
Since $\ep$ and $\eta$ can be arbitrary small, 
it follows that $a\lesssim b.$

The case that {{$0$ is an isolated point  of ${\rm sp}(b)\cup\{0\}$}}
can be reduced to the case above by applying, 
for example, Lemma 3.1 of \cite{HO}.  

{{To show that $W(A)$ is almost unperforated,  suppose that $k\la a\ra \le (k-1)\la b\ra.$
Then $k\la a\ra \le k\la b\ra.$ From what has been proved, $\la a\ra \le \la b\ra.$}}

 By Corollary 5.1 of 
\cite{Rordam-2004-Jiang-Su-stable}
(see also
Proposition {{4.9}} of 
\cite{FLII}
{{as well as the end of Definition \ref{Dcuntz} and Definition \ref{Dqtr} of the current paper}}), $A$ has strict comparison
(or $A$ is purely infinite). 
\end{proof}
\begin{rem}
It is worth noting that Theorem \ref{TAD-SC} has also been independently proved in \cite{CLS} with a different point of view. More precisely, \cite[Theorem 3.2]{CLS} shows that (non-unital) $\sigma$-unital simple \CA s with (TAD-2) have strict comparison.
\end{rem}

\begin{cor}
\label{SC-TSR}
Let $A$ be a unital {{stably finite}} 
simple \CA\, with the property (TAD).
Then $A$ has strict comparison and has stable rank one.
\end{cor}
\begin{proof}
This is a corollary of {{Theorem}} \ref{TAD-SC} and \cite[Theorem 5.7]{FLII}. 
\end{proof}

\vspace{0.2in}

In Section 6,  we will show that the condition that 
$A$ is unital in Corollary \ref{SC-TSR} can be removed,
if we additionally assume that $A$ is separable.

\vspace{0.1in}

Recall that, for $x,y\in {\rm Cu}(A),$
we write $x\ll y,$ if for any increasing sequence $\{y_n\}$ with 
$y\leq \sup\{y_n\},$ there exists $n_0$ such that $x\leq y_{n_0}.$

The following property is introduced by  L. Robert (\cite{Rl1}).
It may be viewed as a tracial version of almost divisibility which is closely related to 
Winter's tracial
0-almost divisibility
(\cite[Definition 3.5]{W-2012-pure-algebras}). 

\begin{df}[Proposition 6.2.1 of \cite{Rl1}]\label{DCD}
Let $A$ be a \CA. We say that ${\rm Cu}(A)$ has property (D), if for any $x\in {\rm  Cu}(A),$ $x'\ll x,$ and any integer $n\in  \N,$
there exists $y\in {\rm Cu}(A)$ such that $n \widehat{y}\le \widehat{x}$ and $\widehat{x'}\leq (n+1) \widehat{y}$
(see Definition \ref{DLAFFQT} for $\widehat z,$
$z\in {\rm Cu}(A)$).
\end{df}

From Corollary 5.8 of \cite{ERS}
 as observed by L. Robert,
he shows the following
(in the proof of Proposition 6.2.1 of \cite{Rl1}).

\begin{lem}\label{LRobert}
Let $A$ be a  finite simple 
\CA\, with strict comparison.  Suppose that $A$ has property (D).
Then the canonical map from $Cu(A)$ to $\LAff_+({\widetilde{ QT}}(A))$ is surjective.
\end{lem}

\begin{proof}
The proof is contained in the second paragraph of the proof of Proposition 6.2.1 of \cite{Rl1}.
\end{proof}

The following is an analogue of  Theorem 6.6 of \cite{ERS}. 
Recall 
(\cite[Proposition 6.4 (iv)]{ERS}) that, in a simple 
\CA\,$A,$  every element of ${\rm Cu}(A)$ is purely non-compact except 
for the elements $[p],$  where $p$  
is {{a non-zero}} finite projection. In particular, if $A$ has no infinite projections,
then the set of purely non-compact elements of ${\rm Cu}(A)$ is precisely 
those elements which cannot be represented by a projection.

\begin{thm}\label{TLAFF}
Let $A$ be a 
non-elementary separable simple 
 \CA\, which is tracially approximately divisible.
Then the map $\la a\ra \to \widehat{\la a\ra}$ is an isomorphism 
{{between}} ordered {{semigroups}} of purely non-compact  elements of ${\rm Cu}(A)$ and of 
$\LAff_+({\widetilde{ QT}}(A)).$

\end{thm}

\begin{proof}
{{If $A$ is purely infinite, then, $\widetilde{QT}(A)=\{0\},$ and every element in $A$ is purely non-compact,
and, all non-zero elements are Cuntz-equivalent. So, in this case, the conclusion 
uninterestingly holds.
Now we assume that $A$ is not purely infinite.
Recall, from Theorem \ref{TAD-SC},  $W(A)$ is almost unperforated. 
By Corollary 5.1 of \cite{Rordam-2004-Jiang-Su-stable} (see also Proposition 4.9 of \cite{FLII}), 
$A$ is stably finite. Consequently, $A$ has no infinite projections.}} 

{{Thus, from  now on in  this proof, we assume that  
purely non-compact elements are precisely those which cannot be represented by projections.}}

Let  $a, b\in (A\otimes {\cal K})_+$ such that $\la a\ra$ and $\la b \ra$ are two  
purely non-compact elements and $d_\tau(a)=d_\tau(b)$ for all $\tau\in {\widetilde{QT}}(A).$
Then, since $A$ is simple
{{and $0$ is not an isolated point of ${\rm sp}(a)\cup\{0\}$}}, 
 for any $\ep>0,$ 
\beq
d_\tau((a-\ep)_+)<{{d_\tau(a)}}=d_\tau(b)\rforal \tau\in {{{\widetilde{QT}}(A)\setminus \{0\}.}}
\eneq
Hence, by {{Theorem}} \ref{TAD-SC} and Theorem \ref{TTAD==}, we have 
$(a-\ep)_+\lesssim b$ for any $\ep>0.$   It follows that 
\beq
\label{TAD-D-F811-2}
a\lesssim b.
\eneq
Symmetrically,  $b\lesssim a.$ So $a\sim b.$
This proves the map $\la a\ra \to \widehat{\la a\ra}$ is injective.

To prove the surjectivity, by the first paragraph of the proof of  Lemma 6.5 of \cite{ERS}, 
it suffices to show that the canonical map $\la a\ra \mapsto \widehat{a}$ is surjective from ${\rm Cu}(A)$ to $\LAff_+({\widetilde{ QT}}(A)).$
Therefore, by {{Lemma \ref{LRobert}}}, 
 it suffices to show that 
$A$ has property (D).

To see $A$ has property (D), let $x, x'\in {\rm Cu}(A)$ such that $x'\ll x.$
Let $a, b\in {{(A\otimes {\cal K})^{\bf 1}_+}}$ such that 
$\la a\ra =x'$ and $\la b\ra =x.$   
By Remark \ref{R1} and \cite[Proposition 5.3]{FLII}, $A\otimes {\cal K}$ is tracially approximately divisible.

Then, for some $1/16>\ep>0,$ 
\beq
\label{TAD-D-F811-3}
a\lesssim f_{2\ep}(b).
\eneq
In particular, we assume that $(b-\ep)_+\not=0.$ 
Note that $f_{\ep/128}(b)\ll b.$ 

Choose $0<\eta<\ep$ such that $(b-\eta)_+\not=0.$ 
Since $A$ is a non-elementary  simple \CA,  there are 
$n+1$ mutually orthogonal and mutually Cuntz-equivalent 
elements $s_1, s_2,...,s_{n+1}\in {\rm Her}((b-\eta)_+)\setminus \{0\}.$ 
Since $A\otimes {\cal K}$ has property (TAD),
there are $d_0, d_1\in {{(A\otimes {\cal K})^{\bf 1}_+}}$
and a \SCA\, $D\otimes M_n\subset  A\otimes {\cal K}$ such that

(1) $b\approx_{\eta/64} d_0+d_1,$

(2) $d_1 
\in_{\eta/64} D\otimes 1_{n},$ and 

(3) $d_0\lesssim s_1.$

Choose $d\in D_+$ such that
\beq\label{LAFF-c-9}
d_1\approx_{\eta/64} \sum_{i=1}^{n}d\otimes e_{i,i}.
\eneq
Then 
\beq
(d_1-\eta/64)_+\approx_{\eta/32} \sum_{i=1}^{n}d\otimes e_{i,i}\andeqn 
(d_1-\eta/32)_+\approx_{\eta/16} \sum_{i=1}^{n}(d-\eta/64)_+\otimes e_{i,i}.
\eneq
By applying {{\cite[Proposition 2.2]{Rordam-1992-UHF2}, we have}}
\beq\label{LAFF-c-10}
\sum_{i=1}^{n}(d-\eta/32)_+\otimes e_{i,i}\lesssim (d_1-\eta/64)_+\andeqn (d_1-3\eta/16)_+\lesssim \sum_{i=1}^{n}(d-\eta/64)_+\otimes e_{i,i}.
\eneq
From (1) above and by \cite[Lemma 1.7]{Phi},
\beq
(d_1-\eta/64)_+\lesssim ((d_0+d_1)-\eta/64)_+\lesssim b.
\eneq
Put $y:=(d-\eta/32)_+\otimes e_{1,1}.$    Then, for all $\tau\in {\widetilde{QT}}(A),$ by  the first inequality in \eqref{LAFF-c-10},
\beq\label{LAFF-c-12}
n d_\tau(y)\le d_\tau((d_1-\eta/64)_+)\le d_\tau(((d_0+d_1)-\eta/64)_+)\le 
d_\tau(b).
\eneq

On the other hand, by (1) and \eqref{LAFF-c-9},
we have
\beq
b\approx_{\eta/32} d_0+\sum_{i=1}^n (d\otimes e_{i,i})\approx_{\eta/32} d_0+\sum_{i=1}^ny\otimes e_{i,i}.
\eneq
It follows that
\beq
\label{TAD-D-F811-1}
(b-\eta/16)_+\lesssim   d_0+\sum_{i=1}^n y\otimes e_{i,i} 
\lesssim s_1\oplus \sum_{i=1}^n y\otimes e_{i,i}.
\eneq
Recall that $(b-\eta/16)_+\in {\rm Ped}(A).$ So 
$d_\tau((b-\eta/16)_+)<\infty$ for all ${\widetilde{QT}}(A).$ 
It follows from  
\eqref{TAD-D-F811-1}, (3), and the choice of $s_1$
 that
\beq
&&\hspace{-0.5in}nd_\tau(y)\ge d_\tau((b-\eta/16)_+)-d_\tau(s_1)
\ge d_\tau((b-\eta/16)_+)-{d_\tau((b-\eta/16)_+)\over{n+1}}\\
&&\ge ({n\over{n+1}})d_\tau((b-\eta/16)_+)\hspace{0.4in} \rforal \tau\in {\widetilde{QT}}(A).
\eneq
In other words,
\beq
\la (b-\eta/16)_+{\widehat{\ra}}\le (n+1)\widehat{y}.
\eneq
By \eqref{TAD-D-F811-3},
$a\lesssim f_{2\ep}(b) 
\lesssim (b-\ep/16)_+.$
It follows that (recall $\eta<\ep$)
\beq\nonumber
\widehat{x'}=\widehat{\la a\ra}\le  \la (b-\eta/16)_+{\widehat{\ra}}\le (n+1)\widehat{y}.
\eneq
Combining this with \eqref{LAFF-c-12}, we conclude that $A$ has property (D) as desired. 
\end{proof}

\begin{rem}\label{RCuntz}
In Theorem \ref{TLAFF}, 
we may write 
$${\rm Cu}(A)=(V(A)\setminus \{0\})\sqcup {\rm LAff}_+({\widetilde{QT}}(A)).$$
Note that,  here,  $0\in  {\rm LAff}_+({\widetilde{QT}}(A))$ is the zero 
element, and,  if $[p]\in V(A)$ and $z\in {\rm LAff}_+({\widetilde{QT}}(A))\setminus \{0\},$ then  
$[p]+z=\widehat{p}+z\in  {\rm LAff}_+({\widetilde{QT}}(A)).$
Moreover,  for 
$x=\la a\ra$ and $y=\la b\ra,$ then $x<y$ if  and only if $\widehat{\la a\ra }<\widehat{\la b\ra }$
(see also Corollary 8.12 of \cite{T20} {{for the unital case}}).
\end{rem}

\section{Stable rank one}

\begin{lem}
\label{spsumnil-lem}
Let $A$ be a \CA, $a,b,c\in A.$ 
Assume $c^n=0$ for some $n\geq 1,$
and  $ab=ac=ca=cb=b^2=0.$
  Then 
  ${\rm sp}((a+b+c)^{n+1})\backslash\{0\}
={\rm sp}(a^{n+1})\backslash\{0\}.$
\end{lem}
\begin{proof}
{{We first claim that
$(a+b+c)^{n+1}=a^{n+1}+ba^n.$}}

{{To see this, let $x$ be a non-zero  term}} in 
the expansion of $(a+b+c)^{n+1}.$  
Note that $x$ is a product of factors $a, b,$ and $c.$ 

Case 1:  $x=a\cdot y,$ where $y$ is a product of $n$ elements in $\{a,b, c\}.$
If  $y$ 
has a factor  $b$ or $c,$ 
then $x$ must have a factor $ab$ or $ac.$ 
{{Then,}} by the assumption $ac=ab=0,$   
$x$ would be zero.
Therefore $y$ has no factor $b$ or $c.$ {{Consequently,}}
$x=a^{n+1}.$

Case 2:   $x=b\cdot y.$

\quad Case 2.1: If $y=a\cdot z,$ where $z$ is still a product of elements in $\{a, b, c\}.$ 
Since,  again, 
 $ac=ab=0$, 
$z$ could not have a factor $b$ or $c.$ 
Therefore the only possible  non-zero $x,$ in this case,
must be $ba^{n}.$

\quad Case 2.2:   $y=b\cdot z.$ This actually is impossible, since $b^2=0.$

\quad Case 2.3:   $y=c\cdot z.$
Then, by the assumption, $ca=cb=0,$ 
$z$ could not have a factor $a$ or $b.$
Thus $x=bc^n.$ 
However, by the assumption $c^n=0,$  
Case 2.3 will not occur. 

Case 3:   $x=c\cdot y.$ 
If $y$ contains factor $a$ or $b,$ then $x$ contains 
factor $ca$ or $cb.$
{{Then,}} by the assumption $ca=cb=0,$ we have $x=0.$
Hence $y$ could not contain factor $a$ or $b.$
Hence the only possible  non-zero $x,$ in this case, must be $c^{n+1}.$
However, by the assumption 
$c^{n+1}=0,$ 
this case could not occur. 

Thus, it leaves two terms: $x=a^{n+1}$ and $x=ba^n.$  In other words, 
$(a+b+c)^{n+1}=a^{n+1}+ba^n.$  This proves the claim.

As $a^nb=0$, we see that
\beq
{\rm sp}((a+b+c)^{n+1})\backslash\{0\}
={\rm sp}((a+b)a^n)\backslash\{0\}
={\rm sp}(a^n(a+b))\backslash\{0\}
={\rm sp}(a^{n+1})\backslash\{0\}.
\eneq

\end{proof}

\begin{cor}
\label{Cor-positive-inv}
Let $a,b,c$ be as in Lemma \ref{spsumnil-lem}.
If, in addition, $a$ is positive, then 
$a+b+c$ can be approximated by invertible elements in $\wtd A.$
\end{cor}
\begin{proof}
By Lemma \ref{spsumnil-lem} and 
the fact that $a\in A_+,$
we have 
${\rm sp}((a+b+c)^{n+1})\backslash\{0\}
={\rm sp}(a^{n+1})\backslash\{0\}\subset \mathbb{R}_+.$
By {{the}} spectral mapping theorem, 
${\rm sp}(a+b+c)\backslash\{0\}$ lies in the union of $n+1$ rays, 
which is
\beq
\{e^{t+2ik\pi /(n+1)}:t\in\mathbb{R}, k=1,...,n+1\}.
\eneq
Hence $0$ is not an {{interior}}  point of 
${\rm sp}(a+b+c)\cup\{0\}.$
Therefore $a+b+c$ can be approximated by invertible elements.
\end{proof}


Recall the the definition of continuous scale from 
\cite{Lin1991}. 

\begin{df}
\cite[Definition 2.5]{Lin1991}
Let $A$ be a 
$\sigma$-unital, non-unital and 
non-elementary 
simple
\CA. 
$A$ {{is said to have}} continuous scale, 
if there is an {{approximate}} identity 
$\{e_n\}$
such that,
for any 
$a\in A_+\backslash\{0\},$
there is $N\in\mathbb{N},$
such that for any $n>m\geq N,$
$e_n-e_m\lesssim a.$
\end{df}

Note that, by Theorem 2.8 of \cite{Lin1991} and Theorem 2.4 of \cite{Lin04}, $A$ has continuous scale if and only if the corona algebra $M(A)/A$ is 
simple. 
It also follows from Theorem 2.4 of \cite{Lin04} that 
if $A$ has continuous scale, then any approximate 
identity $\{e_n\}$ with $e_{n+1}e_n=e_ne_{n+1}=e_n$ for all $n\in \N$ has the property that, for any $a\in  A_+\setminus \{0\},$
there exists $N\ge 1$ such that, for all $m>n\ge N,$ $e_m-e_n\lesssim a.$ 

Moreover, 
for a $\sigma$-unital non-elementary
simple exact \CA\, $A$ with
$T(A)\not=\emptyset,$  
if $A$ has strict comparison,
then 
$A$ has continuous scale if and only if 
 $T(A)$ is compact (see  Proposition 5.4 of \cite{EGLN}, see also the proof of Theorem 5.3  of \cite{EGLN}, and, 
 an early version, Proposition 2.2 of \cite{Lin04}).  See also the third paragraph of the proof of Corollary \ref{Ctadsr1}.

\begin{thm}
\label{TAD-TSR1}
Let $A$ be a  
$\sigma$-unital  projectionless simple 
\CA\ {{with}} continuous scale. 
Suppose that,
for every $\sigma$-unital hereditary $C^*$-subalgebra $B\subset A,$
any non-invertible element in $B$ can be approximated (in norm) by products of two 
nilpotent elements in $B.$ 
Then $A$ has stable rank one. 
\end{thm}
\begin{proof}


Let $\{e^{(k)}\}$ be an approximate identity of $A$
such that $e^{(k+1)}e^{(k)}=e^{(k)}e^{(k+1)}
=e^{(k)}.$ By passing to a subsequence, \wilog, we may assume 
$e^{(k_2)}-e^{(k_1)}\neq 0$ 
for all $k_2 > k_1\in \mathbb{N}.$
Let $\wtd A$ be the 
unitization of $A$ {{and}}
$y+\lambda\in \wtd A,$  {{where $y\in A$ and 
$\lambda\in \C.$}}

Fix $\dt\in (0,1).$ To show $A$ has stable rank one, it is 
suffices to show the following:

Goal:  there is $z\in GL(\wtd A)$ such that
\beq\label{Lsr1-1}
\|(y+\lambda)-z\|<\dt.
\eneq



To achieve the goal, we first choose $k_0\in\mathbb{N}$ such that
$
e^{(k_0)}ye^{(k_0)}
{{\approx_{\dt/4}}} y.
$
Put 
$B:={\rm Her}_A(e^{(k_0)})$ 
and
$y_0=e^{(k_0)}ye^{(k_0)}\in B.$
Then
\beq
\label{TAD-TSR-807-1}
\|(y_0+\lambda)-(y+\lambda)\| <\dt/4.
\eneq

Note, since $A$ has no non-zero projection, $B$ is non-unital. 
 Put $A_1:=\C \cdot 1_{{{\wtd A}}}+B\cong \wtd B.$  
 Let $v\in A_1$ be such that $v^*v=1.$ 
Then $1-vv^*\in A$ is a projection. 
Since  $A$ has no non-zero projections, $v^*v-vv^*=0.$
In other words, {{$A_1$}} is finite and hence
every one-sided invertible element in 
$A_1$ is invertible. 
%
If $y_0+\lambda\in GL(A_1)\subset GL(\wtd A),$
 then, by \eqref{TAD-TSR-807-1}, 
 our goal is achieved. 
 So we assume that $y_0+\lambda \not\in GL(A_1).$ 
 By  \cite[Proposition 3.2]{Rordam-1991-UHF}, 
 there is a two-sided zero {{divisor}} $y_1+\lambda_1\in A_1$
 such that $\|y+\lambda-(y_1+\lambda_1)\|<\dt/4,$  {{for some  $y_1\in B$ and $\lambda_1\in \C.$}}
Therefore, to achieve the goal above, 
it is suffices to 
show  that
$y_1+\lambda_1\in\overline{GL(\wtd A)}.$ 


By  \cite[Lemma 3.5]{Rordam-1991-UHF}, 
 working in $A_1\cong \wtd B,$ 
 we 
can find 
 $a_1\in {A_1}_+\setminus \{0\}$  
 and 
 a unitary $u\in A_1$ such that 
 \beq
 \label{TAD-SR1-806-1}
 a_1u(y_1+\lambda_1)=u(y_1+\lambda_1)a_1=0.
 \eneq
Since $B$ is an essential ideal of $A_1,$ 
there is $b\in {B}_+$ such that $a_1ba_1\not=0.$ 
 Put $a=a_1ba_1\in B,$
and we may assume $\|a\|=1.$ 
 We write $u(y_1+\lambda_1)=x_1+\eta,$ where $x_1\in B$ and 
 $|\eta|=|\lambda_1|.$
{{Since $u$ is invertible, to show $y_1+\lambda_1\in\overline{GL(\wtd A)},$
 it suffices to show that $x_1+\eta\in \overline{GL(\wtd A)}.$}}
 
If $\eta=0,$ then $\lambda_1=0,$ 
and then $y_1+\lambda_1=y_1\in B\subset A.$ 
{{By}} our assumption, $y_1$ 
can be approximated by products of two nilpotent elements in $A,$
which 
can be approximated by invertible elements in $\wtd A. $ 
Thus we may assume that $\eta\neq 0.$

Put $x=\frac{x_1}{\eta}\in \Her_A(e^{(k_0)}).$
If $x+1=\frac{x_1}{\eta}+1\in \overline{GL(\wtd A)},$
then $x_1+\eta\in\overline{GL(\wtd A)}.$
 
Hence,
it is suffices to 
show   that  $x+1\in\overline{GL(\wtd A)}.$ 


 



To do that, let us fix $\ep>0.$
As $A$ has continuous scale, we may 
choose $e_0, e_1,   
\in A_+^{\bf 1}$ 
of the form $e^{(k)}$ with $k>k_0,$
such that $e_0 e_1=  e_1=e_1e_0,$  
$e_0-e_1\lesssim a.$
Note, since 
$x\in \Her_A(e^{(k_0)}),$ 
and $e_1$ is of the form 
$e^{(k)}$ with $k>k_0,$
we also have 
\beq
e_1 x=x=xe_1\,\,\,
{\text{and,\,\,\,hence}}\,\,\, e_1(x+1)=(x+1)e_1.
\eneq
Note, by \eqref{TAD-SR1-806-1} and the choice of $a$ and $x+1,$
we have $a(x+1)=(x+1)a=0.$
We also have
\beq
x+1&=&
(1- e_0)+( e_0- e_1)+ e_1(x+1).
\label{TAD-TSR1-6}
\eneq
%
Since $e_0-e_1 \lesssim a\sim a^2,$ there is $r\in A$ such that 
\beq
\label{TAD-TSR1-1}
(e_0-e_1-\ep)_+= r^*a^2r.
\eneq
Note that, since $(e_0-e_1)e^{(k_0)}=e^{(k_0)}(e_0-e_1)=0,$
and $a\in \Her_A(e^{(k_0)}),$
we have 
\beq
\label{TAD-TSR1-5-1}
(ar)^2=0.
\eneq

Let $C=\{z\in A: za=az=0\}.$ Then $C$ is a hereditary 
{{\SCA\,}} of $A.$
{{Since}} $e_1$ commutes with $x+1$ and $a(x+1)=(x+1)a=0,$ 
we have 
$ae_1(x+1)=a(x+1)e_1=0,$
and $e_1(x+1)a=0.$
Thus $e_1(x+1)\in C.$ 
Let $D=\Her_A (e_1(x+1))\subset C,$
{{which is a $\sigma$-unital hereditary $C^*$-subalgebra.}} 
Note that, since $D$ is  projectionless, $e_1(x+1)$ is not invertible in $D.$
By the assumption of the theorem, 
there are two nilpotent elements $s_1,s_2\in D$
such that 
\beq
\label{TAD-TSR1-2}
s_1s_2\approx_\ep e_1(x+1).
\eneq
Since $e_0e_1=e_1,$ and $D\subset \Her_A(e_1),$
we have 
\beq
\label{TAD-TSR1-4}
s_1(1-e_0)=(1-e_0)s_1=0=
(1-e_0)s_2=s_2(1-e_0).
\eneq
Also note that 
(recall, $s_1,s_2\in D\subset C=\{a\}^{\bot}$)
\beq
\label{TAD-TSR1-3}
{{
(1-e_0)a=0=a(1-e_0),  
\quad
as_1=s_1a=0,
\quad 
as_2=s_2a=0.}} 
\eneq
Then 
by \eqref{TAD-TSR1-6}, 
\eqref{TAD-TSR1-1}, 
\eqref{TAD-TSR1-2}, 
\eqref{TAD-TSR1-4}, 
and \eqref{TAD-TSR1-3}, 
\beq
x+1
&=&
(1-e_0)+(e_0-e_1)+e_1(x+1)
\\
&\approx_{2\ep}&
(1-e_0)+(e_0-e_1-\ep)_++s_1s_2
\\&=&
(1-e_0)+r^*a^2r+s_1s_2
\\&=&
((1- e_0)^{1/2}+r^*a+s_1)((1- e_0)^{1/2}+ar+s_2).
\label{decommul-1}
\eneq

Let $\alpha=(1-e_0)^{1/2},$
$\beta=ar,$ $\gamma =s_2.$
Then 
$(1- e_0)^{1/2}+ar+s_2=\alpha+\beta+\gamma.$
{{Note that}}
$\alpha$ is a positive element, 
$\gamma$ is a nilpotent element, 
and by \eqref{TAD-TSR1-3}, 
$\alpha\beta=\gamma\beta=0,$
by \eqref{TAD-TSR1-4},
$\alpha\gamma=\gamma\alpha=0,$
by \eqref{TAD-TSR1-5-1}, 
$\beta^2=0.$
Then by Lemma \ref{spsumnil-lem} and Corollary 
\ref{Cor-positive-inv}, 
$\alpha+\beta+\gamma=(1- e_0)^{1/2}+ar+s_2$ can be approximated by invertible elements in $\wtd A.$

{{The same}} argument 
also holds for 
$(1- e_0)^{1/2}+ar+s_1^*=((1- e_0)^{1/2}+r^*a+s_1)^*.$ {{Thus}} 
$(1- e_0)^{1/2}+r^*a+s_1$ also can be approximated by invertible elements in $\wtd A.$
{{Then  by \eqref{decommul-1},  we obtain 
$z'\in GL(\td A)$ such that
$\|(x+1)-z'\|<2\ep.$
Since $\ep$ is arbitrary,   this implies that
\beq
x+1\in {{\overline{GL(\wtd A)}}}
\eneq
 as desired.}}
Therefore  $A$ has stable rank one. 
\end{proof}

We have the following dichotomy 
 for
separable simple 
tracially approximately divisible $C^*$-algebras.

\begin{cor}\label{Ctadsr1}
Let $A$ be a 
separable simple 
\CA\  which is tracially approximately divisible.
Then either $A$ is purely infinite or 
$A$ has stable rank one.
\end{cor}
\begin{proof}
We assume that $A$ is not purely infinite. By Theorem \ref{TAD-SC} and Corollary  5.1 of 
{{\cite{Rordam-2004-Jiang-Su-stable}}}
(see {{also}} Proposition 4.9 
of \cite{FLII}), $A$ is stably finite. 
So from now on we will assume that $A$ is stably finite.

We will 
use the fact that every hereditary \SCA\, of $A$  is tracially approximately divisible
(by  Theorem 5.5 of \cite{FLII}). 
Suppose that  $A$ contains a non-zero projection $p,$ 
then  the unital hereditary $C^*$-subalgebra 
$pAp$ has stable rank one, by Corollary \ref{SC-TSR}. 
It follows that $A$ also has stable rank one. 

We now assume that $A$ is projectionless.

By Theorem \ref{TLAFF}, we can choose 
 $e\in {\rm Ped}(A)_+\setminus\{0\}$ such that $\widehat{\la e\ra}$ is continuous on ${\widetilde{QT}}(A).$
Let $A_0=\overline{eAe}$ and 
let 
$QT(A_0)=\{\tau\in {\widetilde{QT}}(A): d_\tau(e)=1\}.$   
Recall  that $\widehat{\la e\ra}(=d_\tau(e))$ is continuous on ${\widetilde{QT}}(A).$ 
It follows that $QT(A_0)$ is compact (see also the last paragraph  of   Definition \ref{Dqtr}). 
We claim that $A_0$ has continuous scale (see the proof of Proposition 5.4 of \cite{EGLN}).  
Indeed, let $\{e_n\}$ be an approximate identity with property that $e_{n+1}e_n{{=e_ne_{n+1}}}=e_n.$
Then ${\widehat{\la e_n\ra}}$ converges uniformly on $QT(A_0).$   
Therefore, by strict comparison,
for any $a\in A_+\setminus \{0\},$ there exists $N\ge 1$ such that, 
for all $m>n\ge N,$
\beq
e_m-e_n\lesssim a.
\eneq
This proves the claim. 

Next we claim that, by the proof of \cite[Theorem 5.7]{FLII},
every element in a projectionless simple \CA\, which is tracially approximately divisible can be approximated by 
the products of two nilpotent elements.  
To see this,  let $x'\in A.$ 
Since $A$ is a 
non-unital   separable \CA, 
for any $\ep>0,$ there are $a\in A_+\setminus \{0\}$ and $x\in A$ such that
$x\approx_\ep x'$ and $ax=xa=0.$
It then suffices to show that $x$ can be approximated by products of two nilpotents.
Then the proof of 
 \cite[Theorem 5.7]{FLII} from  the second paragraph can be applied. 
 Note that   in the last estimate (e 5.29)  at the end of that proof,  $v$ and $w$ are nilpotents.
 This proves the claim. 
 
As we pointed out at the beginning of the proof,  
every hereditary \SCA\, of $A_0$ is tracially approximately
divisible. So Theorem \ref{TAD-TSR1} implies that $A_0$ has stable rank one. 
By Theorem  3.6 of \cite{Rie83},   $A_0\otimes {\cal K}$ has stable rank one, so does $A\otimes {\cal K}$ 
(by \cite{Br}).  It follows from Corollary 3.6  of \cite{BP}, $A$ itself has stable rank one.
\end{proof}

\begin{df}\label{DTrapp}
Recall  {{from  Definition 3.1 of \cite{FLII}}}  that a simple \CA\,  
$A$
is essentially tracially in the class of 
${\cal Z}$-stable \CAs,
if for any finite subset ${\cal F}\subset A,$ any $\ep>0,$
any  
$s\in A_+\setminus \{0\},$  
 there exist an element $e\in A_+^{\bf 1}$ and a non-zero  
 \SCA\ $B$ of $A$
 which is in ${\cal Z}$-stable,  such that

(1) 
$\|ex-xe\|<\ep\rforal {{x}}  
\in {\cal F},$

(2)  $(1-e)x\in_{\ep} B$   
and
$\|(1-e)x\|\ge \|x\|-\ep$ for all $x\in {\cal F},$  and

(3)  $e\lesssim s.$

\end{df}

\begin{thm}\label{TTZsr1}
Let $A$ be a separable simple \CA\, which is {{essentially}} tracially in the class of ${\cal Z}$-stable \CA s.
Then $A$ is purely infinite, or $A$ has stable rank one and ${\rm Cu}(A)=(V(A)\setminus \{0\})\sqcup \LAff_+({\widetilde{QT}}(A)).$ 
\end{thm}

\begin{proof}
It follows from Theorem 5.9 of \cite{FLII} that $A$ is tracially approximately divisible.
Then, by Theorem \ref{TLAFF} and Corollary \ref{Ctadsr1}, {{$A$ is purely infinite, or has stable rank one,}} and 
 ${\rm Cu}(A)=(V(A)\setminus \{0\})\sqcup \LAff_+({\widetilde{QT}}(A)).$
\end{proof}

R\o rdam showed that every unital simple ${\cal Z}$-stable \CA\, is either purely infinite, or 
has stable rank one (see \cite{Rordam-2004-Jiang-Su-stable}).  In \cite{Rlz}, L. Robert showed that every stably projectionless simple ${\cal Z}$-stable 
\CA\, has almost stable rank one and left open whether it actually has stable rank one.
The following corollary answers his question affirmatively.

\begin{cor}
\label{Cor-Z-stable-F817}
Let $A$ be a simple ${\cal Z}$-stable  \CA. Then $A$ is either purely infinite or has stable rank one.
\end{cor}

\begin{proof}
If $A$ contains a non-zero projection $p,$ 
then  by \cite[Theorem 6.7]{Rordam-2004-Jiang-Su-stable}, $pAp$ is either purely  infinite, 
or has stable rank one. So Corollary follows by \cite{Br} and \cite{Rie83}.
Therefore we may assume that $A$ is projectionless. 
Let $x+\lambda\in\wtd A,$
where $x\in A, \lambda\in\mathbb{C}.$ 
Let $\ep>0.$
Since ${\cal Z}\cong \bigotimes_{n=1}^\infty{\cal Z}$ 
(see \cite[Corollary 8.8]{JS1999})
and $A$ is ${\cal Z}$-stable, 
there is an isomorphism $\alpha: A\otimes {\cal Z} \to A$
such that $\alpha(x\otimes 1_{\cal Z})\approx_{\ep/2} x$
(see \cite[Lemma 4.4]{Rordam-2004-Jiang-Su-stable}). 
Note that $\alpha$ extends to an isomorphism 
$\wtd \alpha: (A\otimes {\cal Z})\wtd\ \to \wtd A.$

By the fact that $A$ is simple and by using Lemma 3.7 of \cite{EGLN-KK=0} repeatedly,
we obtain a sequence of separable $C^*$-subalgebras 
$C^*(x)\subset B_1\subset B_2\subset...\subset A$ such that 
$B_n$ is full in $B_{n+1}$ for all $n\in\mathbb{N}.$ 
It follows that $B=\overline{\cup_{n\in\mathbb{N}}B_n}$ is 
separable and simple.
Then $B\otimes\cal Z$ is separable, 
simple, 
${\cal Z}$-stable and projectionless. 
By Theorem \ref{TTZsr1}, there is an invertible element $z\in GL((B\otimes {\cal Z})\wtd\ )\subset GL((A\otimes {\cal Z})\wtd\ )$ such that 
$z\approx_{\ep/2}x\otimes 1_{\cal Z}+\lambda.$ 
Then $\wtd \alpha(z)\in GL(\wtd A)$ and 
$\tilde \alpha(z)\approx_{\ep/2}
\wtd \alpha(x\otimes 1_{\cal Z}+\lambda)\approx_{\ep/2} x+\lambda.$
It follows that $A$ has stable rank one. 
\end{proof}

\begin{prop}
\label{TR0-812}
Let $A$ be a unital {{infinite-dimensional}} separable simple  \CA\, with tracial rank zero. 
Then $A$ is {{essentially}}
 tracially  in the class of  ${\cal Z}$-stable \CAs.
\end{prop}
\begin{proof}
Let $\ep>0,$ let ${\cal F}\subset A$ be a finite subset of $A$ and let $a\in A_+^{\bf 1}\setminus \{0\}.$
Since $A$ has tracial rank zero, there is a non-zero projection 
$p\in A$ and a {{finite-dimensional}}  \SCA\, $F\subset A$  with $1_F=p$
such that

 (1) $xp\approx_{\ep/2} px$ for all $x\in {\cal F},$
 
 (2) $pxp\in_{\ep/2} F,$ and 
 
 (3) $1-p\lesssim a.$
 
 Write $F=M_{r(1)}\oplus M_{r(2)}\oplus \dots \oplus M_{r(m)}.$
  Let $\{e_{i,j}^{(k)}\}_{1\le {{i,j}}\le  r(k)}$
 be a system of matrix units for $M_{r(k)},$ $1\le k\le m.$
 By 
{{\cite[Corollary 4.4]{PeRo04},}} 
 for each $k,$ there is a unital simple AF-algebra $B_k$ and unital embedding $\phi_k: B_k\to 
 \overline{e_{1,1}^{(k)}Ae_{1,1}^{(k)}}$ 
 such that $V(\phi): V(B_k)\to V( \overline{e_{1,1}^{(k)}Ae_{1,1}^{(k)}})
 {{\ =V(A)}}$ is surjective. 

{{We claim that $V(A)$ is not finitely generated.
To see this, suppose that $V(A)$ is generated by $[p_1], [p_2],...,[p_m].$
We may assume that $p_i\in M_l(A)$ (for some $l\ge 1$) is a non-zero projection, $1\le i\le m.$ 
Note that $M_l(A)$ is also a unital  infinite-dimensional simple \CA\, of real rank zero. 
By  repeatedly applying Lemma 1.1 of \cite{Zh},
for example,
we obtain a sequence of 
non-zero  projections $\{q_n\}\subset M_l(A)$ such that
$\lim_{n\to\infty} \sup\{\tau(q_n): \tau\in T(A)\}=0.$ 
 Then, for any non-negative integers $k_1, k_2,...,k_m$ (not all zero),  there is an integer $N\ge 1$ such that 
 $\tau(q_N)<\sum_{i=1}^m k_i\tau(p_i)$ for all $\tau\in T(A).$ 
 Hence $[q_N]$ is not in $V(A).$ This proves the claim.}}

Since $V(A)$ can not be finitely generated, we deduce that each $V(B_k)$ is not finitely generated either, and in particular 
 each $B_k$ is infinite-dimensional.
 Define $C_k:=\{\phi_k(b)\otimes e_{i,j}^{(k)}:{1\le i,j\le r(k)}, b\in B_k\}
 \cong B_k\otimes M_{r(k)}$
  ($1\le k\le m$)
 and $C:=\bigoplus_{k=1}^m C_k.$  Then,  $F\subset C.$
 By (2),
 \beq\label{TRR0-4}
 pxp\in_{\ep/2} C.
 \eneq
 Since each $C_k$ is a unital simple infinite-dimensional
 AF-algebra, $C_k$ is ${\cal Z}$-stable, $1\le k\le m$
 {{(see \cite[Corollary 6.3]{JS1999}).}}
 Therefore $C$ is ${\cal Z}$-stable.   By (1), \eqref{TRR0-4} and (3), $A$ is essentially tracially 
 in the class of ${\cal Z}$-stable \CA s. 
 \end{proof}

\begin{exm}
In \cite{NW}, {{Niu and Wang}} constructed a class of 
separable simple 
{{exact}}  
non-nuclear  {{ \CA  s}}  which {{have}}  tracial rank zero
but not ${\cal Z}$-stable.   
Then, by Proposition \ref{TR0-812},
Niu and Wang's examples are
unital 
separable simple
exact \CAs\ which are  
essentially tracially ${\cal Z}$-stable but not ${\cal Z}$-stable.
By  Theorem 5.9 of \cite{FLII}, these \CA s are {{particularly}} tracially approximately divisible.

\end{exm}

\section{Examples}
\begin{exm}\label{REx1}
It is shown in \cite[Theorem 5.9]{FLII} that a simple \CA\, $A$ which is {{essentially tracially}} in ${\cal C}_{{\cal Z}}$ (see \cite[Notation 4.1]{FLII}), 
then $A$ is tracially approximately divisible.  Any simple \CA s  $A^C_z$  constructed in Theorem 8.4
of \cite{FLII}  and any hereditary \SCA\, of $A_z^C$ (by Proposition 3.5 of \cite{FLII})  are tracially approximately divisible. 
Therefore all  (non-unital hereditary \SCA s) of \CA s in Theorem 8.6 of \cite{FLII} are tracially approximately 
divisible and non-nuclear \CA s.  
By Corollary \ref{Ctadsr1}, all these \CA s have stable rank one. 
\end{exm}

Recall that a ${\rm II}_1$ factor $(N,\tau)$ is said to have property $\Gamma,$
if there is a sequence of unitaries $\{u_n\}\subset N$ satisfying $\lim_{n\to\infty}\|u_nx-xu_n{{\|_2}}=0$ for all $x\in N,$ 
and $\tau(u_n)=0$ for all $n\in\mathbb{N}.$

\hspace{0.1in}

The following is  
well-known to experts.

\begin{prop}
\label{TAD-Gamma-F816}
Let $A $ be a unital {{infinite-dimensional}}   separable 
simple \CA\ 
with a unique tracial state $\tau$ which is also tracially approximately  divisible.
Let $\pi_\tau$  be the GNS representation with respect to $\tau,$
and 
$N:=\pi_\tau(A)''$
the weak closure of $\pi_\tau(A).$
Then $(N,\tau)$ is a ${\rm II}_1$ factor with {{property}} $\Gamma.$
\end{prop}

\begin{proof}
Since $\tau$ is an extreme point of $T(A)=\{\tau\},$ 
$N$ is a ${\rm II}_1$ factor (\cite[Theorem 6.7.3]{Dix77}).
From Theorem \ref{TTAD==} we know that if $A$ is tracially approximately divisible then $A$ has property (TAD-3).
Thus there is a unital embedding 
$\hat \psi: M_2\to \pi_{cu}(A)'.$ 
By Proposition \ref{central-surj} 
and projectivity of $C_0((0,1])\otimes M_2,$
there is a \hm\ $\bar \psi: C_0((0,1])\otimes M_2\to \pi_\infty(A)'$
such that $\pi_\infty\circ \bar \psi(\iota\otimes e_{i,j})=\hat \psi(e_{i,j}),$
$1\le i,j\le 1.$
Again, using projectivity of $C_0((0,1])\otimes M_2,$ 
there is a \hm\, 
$ \psi:  C_0((0,1])\otimes M_2\to l^\infty(A)\subset l^\infty(N)$
that lifts $\bar \psi.$ 
We may  represent $\psi$ by a sequence of \hm s
$\psi_n:  C_0((0,1])\otimes M_2\to A\subset N.$
Then $\{\psi_n\}$ satisfies {{the following}}

(1) $\lim_{n\to\infty}\|\psi_n(x)a-a\psi_n(x)\|=0$
for all $x\in 
C_0((0,1])\otimes M_2$ and all $a\in A,$
and 

(2) $\{1_A-\psi_n(\iota\otimes 1_{M_2})\}\in N_{cu}(A).$

\noindent
By Proposition \ref{equiv-ideal-F812-1},
(2) implies 
$\lim_{n\to\infty}\tau(\psi_n(1_{M_2}))=1,$ 
hence
\beq
\label{cent-F-817-1}
\lim_{n\to\infty}\tau(\psi_n({{\iota\otimes e_{1,1}}}))=
\lim_{n\to\infty}\tau(\psi_n({{\iota\otimes  e_{2,2}}}))
=1/2.
\eneq


Let 
$y\in N$  and $\ep>0.$
Let $\|x\|_2=\tau(x^*x)^{1/2}$ for all $x\in N.$
Since $A$ is  
dense in $N$ in the strong operator topology,
there is $z\in A$ such that $\|y-z\|_{2}<\ep/4.$ 
By (1), there is $K\in\mathbb{N}$ such that 
$\|\psi_n({{\iota\otimes e_{i,i}}})z-z\psi_n(
\iota\otimes e_{i,i})\|<\ep/2$ 
for all $n\geq K$ and $i\in\{1,2\}.$
Then 
\beq
\label{cent-F-817-2}
\|\psi_n({{\iota\otimes e_{i,i}}})y-y\psi_n(
\iota\otimes e_{i,i})\|_{2}
\leq
\|\psi_n({{\iota\otimes e_{i,i}}})z
-z\psi_n({{\iota\otimes e_{i,i}}})\|_{2}+ \|y-z\|_{2}
< \epsilon.
\eneq
It follows \eqref{cent-F-817-1} and \eqref{cent-F-817-2}
that $\{\psi_n({{\iota\otimes e_{1,1}}})\}$ 
and 
$\{\psi_n({{\iota\otimes e_{2,2}}})\}$
are two mutually orthogonal 
nontrivial central sequences of $N.$
Therefore $N$ has property $\Gamma$ 
(see, for example \cite[Lemma A.7.3]{SS08}).
\end{proof}

We now present an example of unital 
non-elementary separable simple
exact (but non-nuclear) 
\CA\ 
that has  stable rank one, a unique tracial state,  strict comparison, 
and $0$-almost divisible Cuntz semigroup, 
and contains a unital embedded copy of the Jiang-Su algebra ${\cal Z}$,  
but is not tracially approximately divisible.
\begin{exm}

Let $C_r^*(\mathbb{F}_\infty)$ be the reduced 
group $C^*$-algebra of the free group on countably infinitely many generators. It is well known that $C_r^*(\mathbb{F}_\infty)$ is a unital infinite-dimensional 
separable simple
$C^*$-algebra with a unique tracial state $\tau$. 
It is also well known that $C_r^*(\mathbb{F}_\infty)$ is exact. 
Moreover, $C_r^*(\mathbb{F}_\infty)$ has stable rank one (\cite{DHR}) 
and 
has strict comparison for positive elements 
(see \cite[Proposition~6.3.2]{Rl1}). 
Hence, the Cuntz semigroup of  $C_r^*(\mathbb{F}_\infty)$ is 
almost divisible by 
\cite[Corollary~8.12]{T20}. 
By \cite[Proposition~6.3.1]{Rl1}, 
The Jiang-Su algebra ${\cal Z}$
can be unitally embedded into  $C_r^*(\mathbb{F}_\infty).$ 
On the other hand, the group von Neumann algebra 
$L(\mathbb{F}_\infty)$
does not have property $\Gamma$ (see, for example, \cite[Theorem A.7.2]{SS08}).
It follows from 
Proposition \ref{TAD-Gamma-F816}
that 
$C_r^*(\mathbb{F}_\infty)$ can not be 
tracially approximately divisible.
\end{exm}

From a  recent result of Ma and Wu in \cite{MW20} on groupoid $C^*$-algebras,
let us restate the following.



\begin{thm}{\rm(c.f.\cite[Theorem~9.7]{MW20})}
Let $G$ be a locally compact, 
second countable and 
Hausdorff étale minimal groupoid on a compact metrizable space without isolated points. 
Suppose $G$ is almost elementary. Then $C_r^*(G)$ is unital, 
separable, simple,
tracially ${\cal Z}$-absorbing, and, 
is either purely infinite, or has 
stable rank one. 
\end{thm}

\begin{proof}
By \cite[Theorem~9.7]{MW20}, $C_r^*(G)$ is unital, 
separable, simple
and is tracially ${\cal Z}$-absorbing 
in the sense of \cite{HO}.   
It follows from Theorem \ref{TTAD==} 
($(3)\Rightarrow (2)$)
and Corollary \ref{Ctadsr1}
that $C_r^*(G)$ either has 
stable rank one or is purely infinite.
\end{proof}

We end this section by the following dichotomy result on flow actions (see Section 7 of \cite{Rl1} for examples of both cases).

\begin{thm}
Let $A$ be a separable \CA\ with finite nuclear dimension, 
and let $\af: \mathbb{R}\to {\rm Aut}(A)$
be a  flow  with no $\af$-invariant ideals and with {{finite
Rokhlin dimension}}. 
Then $A\rtimes_\af\mathbb{R}$ 
is either purely infinite or has stable rank one.
\end{thm}
\begin{proof}
By \cite[Theorem 4.5]{HSWW17}, 
$A\rtimes_\af\mathbb{R}$ has finite nuclear dimension. 
Since $A$ is separable and has no $\af$-invariant ideals, and $\af$ has finite Rokhlin dimension,
we deduce that $A\rtimes_\af\mathbb{R}$ is separable and simple (see Corollary 3.12 of \cite{HSWW17}).
Then by \cite[Corollary 8.7]{Tik14}, 
$A\rtimes_\af\mathbb{R}$ is ${\cal Z}$-stable.
Then by Corollary \ref{Cor-Z-stable-F817},
$A\rtimes_\af\mathbb{R}$ 
is either purely infinite {{or}}
has stable rank one. 
\end{proof}

\textsc{Xuanlong Fu}

Department of Mathematics, 
University of Toronto, Toronto, Ontario, M5S 2E4, Canada

E-mail: xuanlongfu@qq.com 



\quad

\textsc{Kang Li}

Department of Mathematics,
KU Leuven,  
Celestijnenlaan 200b-box 2400, 3001 Leuven, Belgium

E-mail: kang.li@kuleuven.be

\quad

\textsc{Huaxin Lin}
 
Department of Mathematics, 
East China Normal University, 
Shanghai, China

and 

Department of Mathematics, University of Oregon, Eugene, OR 97403, USA

E-mail: hlin@uoregon.edu

\end{document}